\documentclass[12pt]{article}
\usepackage[usenames]{color}
\usepackage{amsmath,amssymb,amsthm}
\usepackage{hyperref}
\usepackage{enumerate}
\usepackage[all]{xy}
\xyoption{arc}
\numberwithin{equation}{section}

\newtheorem{thm}{Theorem}[section]
\newtheorem{cor}[thm]{Corollary}
\newtheorem{prop}[thm]{Proposition}
\newtheorem{lem}[thm]{Lemma}
\theoremstyle{definition}
\newtheorem{defn}[thm]{Definition}
\theoremstyle{remark}
\newtheorem{rmk}[thm]{Remark}

\newtheorem{conv}[thm]{Convention}

\def\co{\colon\thinspace}
\newcommand{\mb}[1]{\mathbb{#1}}

\newcommand{\Hom}{\ensuremath{{\rm Hom}}}

\newcommand{\Tor}{\ensuremath{{\rm Tor}}}

\newcommand{\colim}{\ensuremath{\mathop{\rm colim}}}
\newcommand{\hocolim}{\ensuremath{\mathop{\rm hocolim}}}

\newcommand{\overto}{\mathop\rightarrow}
\newcommand{\overfrom}{\mathop\leftarrow}

\newcommand{\Spec}{{\rm Spec}}
\newcommand{\Proj}{{\rm Proj}}
\newcommand{\Spf}{{\rm Spf}}
\newcommand{\tmf}{{\rm tmf}}

\newcommand{\TMF}{{\rm TMF}}

\newcommand{\BPP}[1]{{\rm BP}{\left\langle#1\right\rangle}}
\newcommand{\bxi}{{\bar\xi}}
\newcommand{\btau}{{\bar\tau}}
\newcommand{\multmod}{{\cal M}_{\mb G_m}}
\newcommand{\fmultmod}{{\cal M}_{\widehat{\mb G}_m}}
\newcommand{\psth}{\psi\mbox{-}\theta}

\newcommand{\smsh}[1]{\ensuremath{\mathop{\wedge}_{#1}}}

\newcommand{\comp}[1]{\ensuremath{#1^\wedge}}

\newcommand{\pow}[1]{\left[\!\left[{#1}\right]\!\right]}
\newcommand{\laur}[1]{\left(\!\left({#1}\right)\!\right)}

\title{Strictly commutative realizations of diagrams over the
    Steenrod algebra and topological modular forms at the prime $2$}
\author{Tyler Lawson\thanks{Partially supported by NSF
    grant 0805833 and a fellowship from the Sloan foundation.}, Niko Naumann}

\begin{document}
\maketitle
\begin{abstract}
  Previous work constructed a generalized truncated Brown-Peterson
  spectrum of chromatic height $2$ at the prime $2$ as an ${\cal
    E}_\infty$-ring spectrum, based on the study of elliptic curves
  with level-$3$ structure.  We show that the natural map forgetting
  this level structure induces an ${\cal E}_\infty$-ring map from the
  spectrum of topological modular forms to this truncated
  Brown-Peterson spectrum, and that this orientation fits into a
  diagram of ${\cal E}_\infty$-ring spectra lifting a classical
  diagram of modules over the mod-$2$ Steenrod algebra.
  In an appendix we document how to organize Morava's forms of
$K$-theory into a sheaf of ${\mathcal E}_\infty$-ring spectra.
\end{abstract}

\section{Introduction}\label{sec:intro}

Some of the main applications of modern algebraic topology, including
the development of structured ring spectra
\cite{ekmm,hovey-shipley-smith}, have been to the subject of
algebraic $K$-theory.  These new foundations introduce strictly
associative and commutative ring objects in the category of spectra,
together with their categories of modules.  These provide a large
library of new objects whose algebraic $K$-theory can be calculated
and studied.  These illuminate general phenomena that bear on old
calculations of the algebraic $K$-theory of rings, of simplicial
rings, and of spaces.

Based on computational studies in algebraic $K$-theory, Ausoni and
Rognes \cite[Introduction]{ausoni-rognes} now expect the
existence of a redshift phenomenon (generalizing the Bott-periodic
phenomena appearing in the algebraic $K$-theory of fields) and
initiated a long-term program to study the relationship between
algebraic $K$-theory and the chromatic filtration.  These conjectures
have been supported by their work showing that the algebraic
$K$-theory of complex $K$-theory supports information at chromatic
level $2$.

The next computational steps in such a research program would
involve study of the algebraic $K$-theory of objects at chromatic
level $2$.  Ongoing work of Bruner and Rognes aims to compute the
algebraic $K$-theory of the topological modular forms spectrum
$K_*(\tmf_{(2)})$ \cite{john-slide} and of related spectra such as
$\BPP{2}$.  These computations take place using the machinery of
topological cyclic homology.

This computational work is greatly assisted by the use of higher
multiplicative structures.  For an associative object $R$, the
algebraic $K$-theory $K(R)$ and topological cyclic homology $TC(R)$
are spectra connected by a cyclotomic trace $K(R) \to TC(R)$.
However, if the category of $R$-modules has a symmetric monoidal
structure analogous to the tensor product of modules over a
commutative ring, the algebraic $K$-theory and topological cyclic
homology inherit the structure of ring objects themselves
\cite{elmendorf-mandell}.
Computations in topological
cyclic homology involve numerous spectral sequence calculations, and
these are greatly assisted by the existence of ring structures (or the
data of power operations \cite{bruner-rognes}) and by naturality
arguments.

In addition, many of these computations begin with the B\"okstedt
spectral sequence, which requires information about the mod-$p$
homology of the spectrum in question.

Previous work, based on a study of the moduli of elliptic curves with
level $\Gamma_1(3)$-structures, showed the following result.

\begin{thm}[{\cite[Theorem~1.1]{lawsonnaumann}}]
\label{thm:reminder}
There exists a $2$-local complex oriented $\cal E_\infty$-ring
spectrum $\tmf_1(3)_{(2)}$ such that the composite map of graded rings
\[
\mb Z_{(2)}[v_1, v_2] \subset BP_* \to MU_{(2),*} \to \tmf_1(3)_{(2),*}
\]
is an isomorphism.
\end{thm}
(Here we say that a multiplicative cohomology theory is ``complex
oriented'' if it is given compatible choices of orientation for all
complex vector bundles; we employ the fact that this is equivalent to
the choice of a map of ring spectra $MU \to R$.)

However, this result was obtained by means of obstruction theory, and
only used the modular interpretation of $\tmf_1(3)_{(2)}$ in a
superficial way. The goal of the current paper is to gain a better
understanding of the larger context inhabited by the spectrum
$\tmf_1(3)_{(2)}$; this is closely related to the study made by
Mahowald and Rezk in~\cite{mahowald-rezk}.  With the above
$K$-theoretic applications in mind, another goal is to exhibit the
mod-$2$ cohomology of $\tmf_1(3)_{(2)}$.  (Forthcoming work of Hill and
the first author should recover a $C_2$-action and a connective
spectrum $\tmf_0(3)$.)

There is a map of moduli stacks of generalized elliptic curves
\[
\overline{\cal M}_1(3)\to \overline{\cal M}
\]
(\cite[Theorem 4.1.1, (1) with $N=3$, $n=1$]{conrad}).  This is the
unique map extending the map that takes a smooth elliptic curve with a
$3$-torsion point and forgets the point.  This map ramified at exactly
one of the two cusps of $\overline{\cal M}_1(3)$ but is log-\'etale.
The modular interpretation of $\tmf$ suggests that this map should
have a topological realization.  In fact, we would like to construct a
$2$-local commutative diagram of ${\cal E}_\infty$-ring spectra
corresponding to (the connective covers of) the global sections of
sheaves of ${\cal E}_\infty$-ring spectra in the following diagram:
\[
\xymatrix{ \overline{\cal M} & & [\Spec(\mb Z_{(2)})/\!/ \{ \pm 1\}]  \ar[ll]_(.6){\mathrm{\small cusp}}\\
  \overline{\cal M}_1(3)\ar[u] & & \Spec(\mb
  Z_{(2)})\ar[ll]_{\mathrm{\small ramified\,\, cusp}} \ar[u] }
\]
A realization of this diagram is achieved by the following main
result of this note.

To state it, we recall that, for each $n$, the mod-$2$ Steenrod algebra
  ${\cal A}^*$ contains exterior subalgebras $E(n)$ generated by the Milnor
  primitives $Q^0, \ldots, Q^n$, and larger subalgebras ${\cal
      A}(n)$ generated by $Sq^1, \ldots, Sq^{2^{n+1}}$.

\begin{thm}\label{thm:main}
There is a commutative diagram of connective ${\cal E}_\infty$-ring
    spectra as follows:
\begin{equation}
  \label{eq:main-diagram}
\xymatrix{ \tmf_{(2)} \ar[r]^{c}\ar[d]^{o} & ko_{(2)}\ar[d]^{\iota} \\
\tmf_1(3)_{(2)} \ar[r]^{\tilde{c}} & ku_{(2)}}
\end{equation}
Here $\iota$ is the complexification map, $o$ is a
$\tmf_{(2)}$-orientation of $\tmf_1(3)$, $c$ corresponds to the
cusp on the moduli space of elliptic curves, and $\tilde{c}$
corresponds to the unique ramified cusp on the moduli space of
elliptic curves with level $\Gamma_1(3)$-structure.

In mod-$2$ cohomology, this induces the
following canonical diagram of modules over the mod-$2$ Steenrod
algebra ${\cal A}^*$:
\[ \xymatrix{ {\cal A}^*/\!/{\cal A}(2) & {\cal A}^*/\!/{\cal A}(1)\ar[l]\\
{\cal A}^*/\!/E(2)\ar[u]& {\cal A}^*/\!/E(1).\ar[l]\ar[u]}
\]
There exists a complex orientation of $\tmf_1(3)_{(2)}$ such that
  in homotopy, $\tilde{c}$ induces a map sending the Hazewinkel 
  generators $v_1$ to $v_1$ and $v_2$ to zero, and there is a cofiber
  sequence of $\tmf_1(3)_{(2)}$-modules
\[
\Sigma^6 \tmf_1(3)_{(2)}\stackrel{\cdot v_2}{\longrightarrow} \tmf_1(3)_{(2)}
\stackrel{\tilde{c}}{\longrightarrow} ku_{(2)}.
\]
\end{thm}
We note that the restriction to the ramified cusp of $\overline{{\cal M}}_1(3)$
is not material in the above discussion.  The unramified cusp
also gives rise to a commutative diagram of the same form and with
similar properties.  However, the spectrum in the lower-right corner
is no longer the connective $K$-theory spectrum $ku_{(2)}$ if
multiplicative structure is taken into account, but instead
corresponds to a Galois twist of the multiplicative formal group law
which is defined over $\mb Z[1/3]$.  We discuss the modifications
necessary to use this form of $K$-theory in an appendix. Its periodic
version is most easily described as a homotopy fixed point spectrum:
\[
KU^\tau = (KU \smsh{} \mb S_{(2)}[\omega])^{hC_2}
\]
Here $\mb S_{(2)}[\omega]$ is obtained by adjoining third roots
of unity to the $2$-local sphere spectrum.  The generator of the
cyclic group $C_2$ acts by $\psi^{-1} \smsh{} \sigma$, where
$\psi^{-1}$ is the Adams operation associated to complex conjugation
and $\sigma$ is complex conjugation acting on the roots of unity.

We conclude with a short overview of the content.  In
Section~\ref{sec:elliptic-curve} we collect basic facts about the
moduli of generalized elliptic curves equipped with a
$\Gamma_1(3)$-structure.  In Section~\ref{sec:constructmaps} we
construct the ${\cal E}_\infty$-maps in Theorem~\ref{thm:main} using
Goerss-Hopkins obstruction theory and chromatic fracture squares.
While familiar to the experts, we found it worthwhile to spell out the
details of the $K(1)$-local obstruction theory.  Section
\ref{sec:cohomology} shows $H^*(\tmf_1(3)_{(2)},\mb F_2)\cong {\cal
  A}^*/\!/ E(2)$ by showing more generally that every generalized
$\BPP{n}$ has the same cohomology as the standard $\BPP{n}$
(Definition \ref{def:generalbpn} and Theorem \ref{thm:coh-comp}).
Section \ref{sec:proof} collects these results into a proof of Theorem
\ref{thm:main}.  The appendix discusses forms of $K$-theory and the
changes necessary to realize diagram~(\ref{eq:main-diagram}) using the
unramified cusp rather than the ramified one.

\begin{conv}
With the exception of the appendix, throughout this paper we will
  work in the category of $2$-local spectra.  In particular, the
  names $\tmf$, $\tmf_1(3)$, $ko$, $ku$, and the like denote
  $2$-localizations.
\end{conv}

The authors would like to thank several people: Gerd Laures for
help with the proof of Proposition \ref{prop:TMFtoKO}; Matthew
Ando and Paul Goerss for discussion relating to the
$\tmf$-orientations of $KO$ and $KO\pow{q}$; Andrew Baker for
discussions relating to uniqueness of $\BPP{n}$; Andrew Baker and
Justin Noel for indicating a generalization, and simplification, of
the argument for Theorem~\ref{thm:coh-comp}; John Rognes and
Robert Bruner for motivation and discussions relating to the material
of this paper; and referee for several helpful suggestions.

\section{The $2$-localized moduli $\overline{{\mathcal M}}_1(3)_{(2)}$}\label{sec:elliptic-curve}

This section collects basic facts about the moduli of generalized
elliptic curves with a $\Gamma_1(3)$-structure.  It provides
background for the somewhat ad-hoc construction of \cite[Section
8.2]{lawsonnaumann}, where it was shown that a certain formal
  group law over $\mb Z_{(2)}[A,B]$, to be recalled presently, defines a
$\BPP{2}$-realization problem at the prime $2$ which can be
  solved. We will focus here on determining the ordinary and
supersingular loci of the moduli stack, as those dictate the
chromatic properties of $\tmf_1(3)$.

Consider the curve ${\cal E}\subseteq {\mathbb
  P}^2_{\mathbb{Z}_{(2)}[A,B]}$ defined by the affine Weierstrass equation
\begin{equation}\label{eq:Weierstrass}
y^2 + Axy + By = x^3
\end{equation}
over the graded ring $\mb Z_{(2)}[A,B]$, where $|A| = 1$, $|B| = 3$.
There is a $3$-torsion point at the origin $(0,0)$, and conversely
this Weierstrass form is forced by the requirements that $(0,0)$ is a
$3$-torsion point whose tangent line is horizontal; the elliptic curve
${\cal E}$ is the universal generalized elliptic curve with a choice
of $3$-torsion point (namely $[0:0:1]$) and a choice of invariant
differential $dy/3x^2$ by \cite[Proposition 3.2]{mahowald-rezk}.  The
grading, which can be interpreted as acting on the invariant
differential, is reflected in the action of the multiplicative group
$\mb G_m = \Spec(\mb Z[\lambda^{\pm 1}])$ determined by
\[
A \mapsto \lambda A, B \mapsto \lambda^3 B.
\]
This lifts to a $\mb G_m$-action on ${\cal
  E}/\Spec(\mathbb{Z}_{(2)}[A,B])$ via
\[
(x,y) \mapsto (\lambda^2 x,\lambda^3 y).
\]
The curve ${\cal E}$ has additive reduction only at $(A,B) = (0,0)$.

\begin{prop}
\label{prop:reduction}
The restriction of ${\cal E}$ to $\mb A^2_{\mb Z_{(2)}}\setminus\{ (0,0)\}$
is a generalized elliptic curve with irreducible geometric fibers
as follows:
\begin{enumerate}[i)]
\item A nodal curve of arithmetic genus one if $A^3=27B$ or $B=0$.
\item A supersingular elliptic curve if $A=0$ and $2=0$.
\item An ordinary elliptic curve otherwise.
\end{enumerate}
\end{prop}

This, in particular, expresses $A$ as a lift of the Hasse invariant
$v_1$, which is well-defined mod $2$; the element $B$ is a lift of
$v_2$, which is well-defined mod $(2,v_1)$.

\begin{proof} 
To say that ${\cal E}/(\mb A^2_{\mb Z_{(2)}}\setminus\{ (0,0)\})$
is a generalized elliptic curve with irreducible fibers
in the sense of \cite[Chapitre I, D\'efinition 1.12]{deligne-rapoport} means that the modular
quantities $c_4({\cal E})$ and $\Delta({\cal E})$ have no common zero
on $\mb A^2_{\mb Z_{(2)}} \setminus\{ (0,0)\}$. This results from the
following computations:
\begin{eqnarray*}
c_4({\cal E})&=& A(A^3-24B)\\
\Delta({\cal E})&=&B^3(A^3-27B)\\
j({\cal E}[\Delta^{-1}])&=&\frac{A^3(A^3-24B)^3}{B^3(A^3-27B)}\\
\end{eqnarray*}
A geometric fiber is a nodal curve if and only if $\Delta=0$, so i)
is clear and ii) and iii) follow by recalling that the only
supersingular $j$-invariant in characteristic $2$ is $j=0$.
\end{proof}

The $\mb G_m$-action on ${\cal E}$ over $\mb A^2_{\mb Z_{(2)}} \setminus
\{ (0,0)\}$ descends to determine a generalized elliptic curve over
the quotient stack
\[
\overline{{\cal M}}_1(3)_{(2)} \cong \left[ (\mb A^2_{\mb Z_{(2)}}
  \setminus \{ (0,0)\})\ /\!/\ \mathbb{G}_m\right].
\]
The
stack $\overline{{\cal M}}_1(3)_{(2)}$ is a stack with coarse moduli
isomorphic to the weighted projective space $\Proj(\mb Z_{(2)}[A,B])$.
Since the zero section of a generalized elliptic curve lies in the
smooth locus, we have associated with ${\cal E}$ a 1-dimensional
formal group $\widehat{\cal E}/\overline{{\cal M}}_1(3)_{(2)}$.

The function $-x/y$ is a coordinate in a neighborhood of $\infty$ on
${\cal E}$, and this gives an isomorphism between the pullback of the
relative cotangent bundle $\omega$ of $\widehat{\cal E}/\overline{\cal
M}_1(3)$ along the zero section and the tautological line bundle
${\cal O}(1)$ on $\overline{\cal M}_1(3)$.  Compatibility with the
grading implies that the formal group $\widehat{\cal E}$ comes from a
graded formal group law, and is induced by a map of graded rings $MU_*
\to \mb Z_{(2)}[A,B]$.  Here we follow the standard convention that
elements in algebraic grading $k$ lie in topological grading $2k$.
The elements $A$ and $B$ can then be interpreted as global sections:
\begin{eqnarray*}
A &\in& H^0(\overline{\cal M}_1(3), {\cal O}(1))\\
B &\in& H^0(\overline{\cal M}_1(3), {\cal O}(3))
\end{eqnarray*}
Let $\mb Z_{(2)}[A,B] \to \mb Z_{(2)}[b]$ be the ungraded map given by
$A \mapsto 1$, $B \mapsto b$.  The composite map
\[
\Spec(\mb Z_{(2)}[b]) \to \mb A^2_{\mb Z_{(2)}} \setminus \{ (0,0)\}
\to \overline{\cal M}_1(3)
\]
is an open immersion and thus determines an affine coordinate chart
$V$ of $\overline{\cal M}_1(3)$.  On this chart the elliptic curve
$y^2 + xy + by = x^3$ has no supersingular fibers by Proposition
\ref{prop:reduction},
ii).

\begin{prop}
\label{prop:restriction}
  The restriction of $\widehat{\cal E}$ to $V=\Spec(\mb Z_{(2)}[b])$ 
  is a formal $2$-divisible group whose mod-$2$ reduction has
  constant height $1$.
\end{prop}

\begin{proof}
  By Proposition~\ref{prop:reduction} the restriction
  $\widehat{\cal E}|_V$ is a $1$-dimensional formal group of constant
  height $1$, so it is $2$-divisible.
\end{proof}

Observe that $V\subseteq \overline{\mathcal M}_1(3)$ is the maximal
open substack over which ${\mathcal E}$ is ordinary.

Let $\mb Z_{(2)}[A,B] \to \mb Z_{(2)}[a]$ be the map given by $A
\mapsto a$, $B \mapsto 1$.  The composite map
\[
\Spec(\mb Z_{(2)}[a]) \to \mb A^2_{\mb Z_{(2)}} \setminus \{ (0,0)\}
\to \overline{\cal M}_1(3)
\]
is \'etale.  The stack-theoretic image is the quotient of
$\Spec(\mb Z_{(2)}[a])$ by the action of the group $\mu_3$ of
  third roots of unity, given by $\omega \cdot a  = \omega a$ for
$\omega$ any third root of unity.  The induced map
\[
W= \left[ \Spec(\mb Z_{(2)}[a]) /\!/ \mu_3 \right] \to \overline{\cal
  M}_1(3)
\]
is an open immersion, and so $\Spec(\mb Z_{(2)}[a])$ determines an
\'etale coordinate chart for the stack near $a=0$.  On this chart the
elliptic curve is defined by the Weierstrass equation $y^2 + axy
+ y = x^3$, with $\mu_3$-action given by $\omega \cdot (x,y)
  = (\omega^2 x, y)$.

Let $U$ be the formal scheme $\Spf(\mb Z_{2}\pow{a})$, with formally
\'etale map $U \to \overline{\cal M}_1(3)$.  By Proposition
\ref{prop:reduction}, the pullback of ${\cal E}$ to $U$ has special
fiber a supersingular elliptic curve.  We denote by $\mb G/U$ the
$2$-divisible group of ${\cal E}|_U$.

\begin{prop}
\label{prop:universal-def}
The $2$-divisible group $\mb G/U$ has height $2$
  and is a universal deformation of its special fiber.
\end{prop}

\begin{proof}
This is a restatement of \cite[Proposition 8.2 and Remark 4.2]{lawsonnaumann}.
\end{proof}

The common overlap of the coordinate charts $V$ and $W$ is determined by the
identity $a^3 b = 1$.  The Mayer-Vietoris sequence for this weighted
projective space, using these affine coordinate charts, allows us to
compute the cohomology of $\overline{\cal M}_1(3)_{(2)}$ with
coefficients in ${\cal O}(*)$.

\begin{cor}
\label{cor:moduli-cohomology}
  The cohomology of $\overline{\cal M}_1(3)$ with coefficients in
  ${\cal O}(*)$ vanishes above degree $1$.

  The cohomology groups $H^0(\overline{\cal M}_1(3), {\cal O}(*))$
  form the graded ring $\mb Z_{(2)}[A,B]$.

  The cohomology $H^1(\overline{\cal M}_1(3), {\cal O}(*))$ is the
  module $\mb Z_{(2)}[A,B]/(A^\infty, B^\infty)$ of elements $A^{-n}
  B^{-m} D$, where $D$ is a duality class in $H^1(\overline{\cal
    M}_1(3), {\cal O}(-4))$ annihilated by $A$ and $B$.
\end{cor}

\section{Constructing the maps}\label{sec:constructmaps}

The goal of this section is to construct the ${\cal E}_\infty$-maps of
  Theorem~\ref{thm:main}, diagram~(\ref{eq:main-diagram}),
  which we reproduce here:
\[
\xymatrix{ \tmf \ar[r]^{c}\ar[d]^{o} & ko\ar[d]^{\iota} \\
\tmf_1(3) \ar[r]^{\tilde{c}} & ku}
\]
All spectra appearing in this diagram are the connective covers of
their $K(0)\vee K(1)\vee K(2)$-localizations.  Accordingly, we will
construct the required ${\cal E}_\infty$-maps using two chromatic
fracture squares, followed by taking connective covers.

\subsection{The $K(2)$-local maps}\label{subsec:k2-local-maps}

We identify the $K(2)$-localizations of the connective spectra in
Theorem \ref{thm:main} as follows.  We have that $L_{K(2)}KU\simeq *$,
and as $K(n)$-localization does not distinguish between a spectrum and
its connective cover we have $L_{K(2)}ku \simeq *$ as well. From the
familiar fibration
\begin{equation}\label{eq:eta}
\Sigma ko\stackrel{\eta}{\to} ko\to ku,
\end{equation}
it follows that the nilpotent map $\eta$ induces an equivalence on
$L_{K(2)}ko$, hence $L_{K(2)}ko\cong *$ as well.

Let $E$ denote the Lubin-Tate spectrum associated with the formal
group of the supersingular elliptic curve
\[
C\co y^2+y=x^3
\]
over $\mb F_4$ \cite[Section 7]{goerss-hopkins}.  The group
$G_{48}=Aut(C/\mathbb{F}_4) \rtimes Gal(\mb F_4/\mb F_2)$ acts on $E$
and $L_{K(2)}\tmf\simeq E^{hG_{48}}$ \cite{hopkins-mahowald}.  The
subgroup $\langle\omega\rangle\subseteq Aut(C/\mb F_{4})$ fixing the
point at infinity on $C$ (whose generator sends $(x,y)$ to $(\omega^2
x, y)$) is cyclic of order $3$ and defines a subgroup
\begin{equation}\label{subgroup} 
S_3 \cong \langle\omega\rangle \rtimes Gal(\mb F_4/\mb F_2)\subseteq G_{48}.
\end{equation}
We have $L_{K(2)}\tmf_1(3)\simeq E^{hS_3}$ by the
construction of $\tmf_1(3)$ \cite[proof of Theorem
4.4]{lawsonnaumann}.

We define the $K(2)$-localizations of the maps from
  diagram~(\ref{eq:main-diagram}) as follows:
\[
\xymatrix{ E^{hG_{48}} \ar[r] \ar[d]_{o_{K(2)}} & \ast \ar[d]\\
E^{hS_3}\ar[r] & \ast }
\]
The map $o_{K(2)}$ is defined to be the canonical map of homotopy
fixed point spectra associated with the inclusion of
  equation~(\ref{subgroup}).  

\subsection{The $K(1)$-local maps}\label{subsec:K1-local-maps}

We refer the reader to
\cite{laures,hopkins,ando-hopkins-strickland}, as well as
\cite[Sections 5 and 6]{lawsonnaumann} and references therein,
for an account of basic results about $K(1)$-local ${\cal
  E}_\infty$-ring spectra which we will use freely.

To ease reading, in this subsection only we will abbreviate
\begin{equation}
\label{eq:k1-local-spectra}
\begin{split}
\TMF &=L_{K(1)}\tmf,\\
\TMF_1(3) &=L_{K(1)}\tmf_1(3),\\
KO &=L_{K(1)}ko, \text{ and}\\
K &=L_{K(1)}ku.
\end{split}
\end{equation}
Furthermore, all smash products will implicitly be $K(1)$-localized
and all abelian groups implicitly $2$-completed.  We use $K^\vee_*(-)$
to denote ($K(1)$-localized) $K$-homology, so that
\[
K^\vee_*(-)=\pi_*L_{K(1)}(K\wedge -).
\]
In order to construct the required maps between the ${\cal
E}_\infty$-ring spectra in equation~(\ref{eq:k1-local-spectra}),
we first construct maps between the $\psth$-algebras given by
their $K^\vee$-homology.
\begin{prop}
All spectra in equation~(\ref{eq:k1-local-spectra}) have $K^\vee_*$
concentrated in even degrees, and there are isomorphisms of
$\psth$-algebras as follows. 
\begin{align}
K^\vee_0(\TMF) &\cong V \label{eq:khomoftmf} \\
K^\vee_0(KO) &\cong \Hom_c(\mb Z_2^\times/\{ \pm 1\}, K_0)\label{eq:khomofko}\\
\label{eq:k_homology_of_k} K^\vee_0(K) &\cong \Hom_c(\mb Z_2^\times, K_0)
\end{align}
Here $V$ is (the level $1$-analogue of) Katz's ring of generalized
  modular functions \cite[(1.4.9.1)]{katzhigher}, where it is denoted
  $V_{\infty,\infty}$.

These fit into a commutative diagram of $\psth$-algebras as follows:
\begin{equation}\label{diag:psthdiagram}
\xymatrix{
K_0^\vee(\TMF)\cong V \ar[r]^-a \ar[d]^d &
Hom_c(\mb Z_2^\times/\{ \pm 1\}, K_0)\cong K_0^\vee(KO) \ar[d]^b \\
K_0^\vee(\TMF_1(3)) \ar[r]^-c &
Hom_c(\mb Z_2^\times, K_0)\cong K_0^\vee(K)}
\end{equation}
\end{prop}

\begin{proof}
First, we review the structure of these $\psth$-algebras.
According to \cite[Lemma 1]{hopkins} and \cite[Proposition
3.4]{lauressplitting}, or \cite[Proposition 9.2]{ando-hopkins-rezk},
we have isomorphisms of $\psth$-algebras as follows:
\begin{equation}\label{eq:structure-K-theory}
 K^\vee_*K\cong Hom_c(\mb Z_2^\times, K_*)
\end{equation}
\begin{equation}\label{eq:structure-K-KO}
 K^\vee_*KO\cong Hom_c(\mb Z_2^\times/\{ \pm 1\}, K_*)
\end{equation}
\begin{equation}\label{eq:structure-KO-theory}
KO^\vee_*KO\cong Hom_c(\mb Z_2^\times/\{ \pm 1\}, KO_*)
\end{equation}
Here we have obvious $\mb Z_2^\times$-actions and trivial
$\theta$, in the sense that the ring homomorphism $\psi^2$ is the
  identity.

Furthermore, the inclusion of the constant functions $K_*\subseteq
K_*^\vee K$ (resp. $K_*\subseteq K_*^\vee KO$), as the ring
of $\mb Z_2^\times$-invariants, is a split $\mb Z_2^\times$- (resp.
$\mb Z_2^\times/\{ \pm 1 \}$-) Galois extension.

Next, we consider $K_0^\vee(\TMF_1(3))$.  We have a generalized elliptic
curve
\begin{equation}
\label{eq:ordinary-curve}
E\co y^2+xy+by=x^3
\end{equation}
over $\comp{\mb Z_2[b]}_2$.  By construction \cite[Section
6.2]{lawsonnaumann}, the Hurewicz map
\[
\pi_0\TMF_1(3) \to K^\vee_0\TMF_1(3)
\]
has domain $\comp{\mb Z_2[b]}_2$, and is the $\mb Z_2^\times$-Galois
extension classifying isomorphisms $\widehat{\mb G}_m \to
\widehat{E}$.  We will refer to such an isomorphism as a
trivialization of the ordinary elliptic curve $E/\pi_0(\TMF_1(3))$.
The operation $\theta\co K^\vee_0(\TMF_1(3))\to K^\vee_0(\TMF_1(3))$ is
determined by the canonical subgroup of $E$ \cite[page 35]{gouvea} and
can be computed explicitly \cite[proof of Proposition
8.5]{lawsonnaumann}.  Still by construction, we also have $K_1^\vee
\TMF_1(3)=0$.

The $\psth$-algebra structure on $V$ is determined similarly.  The
ring $V$ carries a universal isomorphism class of trivialization of
its elliptic curve.  It has a continuous action of $\mb Z_2^\times$
(acting on the universal trivialization) and a canonical lift of
Frobenius $\psi^2$, which induces the natural transformation on $V$
determined by the quotient by the canonical subgroup.  Since $V$ is
torsion free, there is a unique self-map $\theta$ of $V$ such that
$\psi^2(x)=x^2+2\theta(x)$ for all $x\in V$. This gives the structure
of a $\psth$-algebra to $V$.

We will now establish the isomorphism of
  equation~(\ref{eq:khomoftmf}).  From \cite[Theorem 3 and
Proposition 1]{laures}, we know $KO^\vee_*\TMF\cong KO_*\otimes
KO^\vee_0 \TMF$ and $V\simeq KO^\vee_0 \TMF$ as $\psth$-algebras.

Since $\TMF$ is equivalent to a ($K(1)$-local) wedge of copies of
  $KO$ \cite[Corollary 3]{laures}, we find that
\begin{equation}\label{eq:khomsplitoftmf}
K^\vee_* \TMF \cong \widehat\oplus K^\vee_* KO.
\end{equation}
Therefore, by equation~(\ref{eq:structure-K-KO}) we find that
$K^\vee_* \TMF$ is concentrated in even degrees.  Moreover,
by equation~(\ref{eq:structure-KO-theory}) we have that the map $V
\cong KO^\vee_0 \TMF \to K^\vee_0 \TMF$ is an isomorphism.

We next construct maps between these $\psth$-algebras as required in diagram~(\ref{diag:psthdiagram}).

\begin{description}
\item{Construction of the map $b$.}
The map $b$ is determined by equations
(\ref{eq:structure-K-theory}), (\ref{eq:structure-K-KO}), and
pull-back along the canonical projection $\mb Z_2^\times\to\mb
Z_2^\times/\{ \pm 1\}$, and is clearly a map of $\psth$-algebras.

\item{Construction of the map $d$.}
By construction there is a trivialization of the elliptic curve
$E$ over $K^\vee_0\TMF_1(3)$.  As $V$ carries the universal
example of a trivialized elliptic curve over a $2$-adically
  complete ring, this determines a map $d\co V=K^\vee_0\TMF\to
  K^\vee_0\TMF_1(3)$.  A $\Gamma_1(3)$ structure on an elliptic curve
  $E$ determines a unique compatible structure on the quotient by the
  canonical subgroup, as $2$ and $3$ are relatively prime.  This
  implies that the induced map of classifying rings is a map of
  $\psth$-algebras.

\item{Construction of the map $a$.}
The map $a$ is constructed in the same manner as $d$: the ring
$K^\vee_0 KO$ carries the Tate curve $y^2 + xy = x^3$ as universal among
isomorphism classes of nodal elliptic curve equipped with a choice
of trivialization, and the map $a\co V \to K^\vee_0 KO$ classifies it.

\item{Construction of the map $c$.}  The map
\[
\pi_0\TMF_1(3)\cong\comp{\mb Z_2[b]}_2\to K^\vee_0\cong\mb
  Z_2,
\]
determined by sending $b\mapsto 0$ specializes the elliptic curve
$E$ of equation~(\ref{eq:ordinary-curve}) over $\pi_0\TMF_1(3)$ to
the Tate curve $T\co y^2+xy=x^3$.  We fix an isomorphism of formal
groups $\widehat{T}\cong\widehat{\mb G}_m$, so that the pullback of
the elliptic curve $E$ under the composite map $\pi_0 \TMF_1(3)\to
K^\vee_0K$ has a trivialization.  By the universal property of
$K^\vee_0 \TMF_1(3)$, this trivialization determines a map $c\co
K^\vee_0\TMF_1(3)\to K^\vee_0K$, and we need to show that it
commutes with $\psi^\alpha$ ($\alpha\in\mb Z_2^\times$) and
$\psi^2$.

Let $(E, f\co \widehat{E}\stackrel{\cong}{\to}\widehat{\mb G}_m)$ be the
universal trivialization of $E$ over $K^\vee_0(\TMF_1(3))$. Then
\[
(\psi^\alpha)_*(E,f)=(E,[\alpha]\circ f)
\]
for any $\alpha\in\mb Z_2^\times=\mathit{Aut}(\widehat{\mb G}_m)$. Hence
\[
(c\circ\psi^\alpha)_*(E,f)= 
(T,c_*([\alpha])\circ c_*(f)) = 
(T,[\alpha]\circ c_*(f)).
\]
On the other hand,
\[
(\psi^\alpha\circ c)_*(E,f)=\psi_*^\alpha(T,c_*(f))=(T,[\alpha]\circ
c_*(f)),
\]
so $\psi^\alpha\circ c=c\circ\psi^\alpha$, making $c$ compatible
with the $\mb Z_2^\times$-action. 

Recall \cite[page 35]{gouvea} the canonical subgroup $C\subseteq E$ is
  defined so that there is the following diagram
of formal groups over $K_0^\vee (\TMF_1(3))$ with exact columns:
\begin{equation}\label{eq:psipcheck}
\xymatrix{ C \ar[rr]^{\cong} \ar@{^(->}[d] & & \mu_2 \ar@{^(->}[d]\\
\widehat{E} \ar[d]\ar[rr]^f_\cong & & \widehat{\mb G}_m \ar[d]_{[2]} \\
\widehat{E/C} \ar[rr]^{\bar{f}}_\cong & & \widehat{\mb G}_m}
\end{equation}

From this, we find that $c_*(\bar{f})=\overline{c_*(f)}$.
By construction of $\psi^2$, we have
\[ (\psi^2)_*(E,f)=(E/C,\bar{f}),\]
By the functoriality of the canonical subgroup and
(\ref{eq:psipcheck}), we therefore find that
\[ (c\circ\psi^2)_*(E,f)=(c_*(E/C),c_*(\bar{f}))=(T/C,\overline{c_*(f)}),\]
On the other hand,
\[ (\psi^2\circ c)_*(E,f)=\psi_*^2(T,c_*(f))=(T/C,\overline{c_*(f)}).\]
Hence $\psi^2\circ c=c\circ\psi^2$, and $c$ is indeed a map of
$\psth$-algebras.
\end{description}

To see that diagram~(\ref{diag:psthdiagram}) commutes, it
suffices to remark that both composites $b\circ a$ and $c\circ d$
classify the same trivialized generalized elliptic curve over
$K^\vee_0K$, and this is true by construction.
\end{proof}

We now start to realize diagram~(\ref{diag:psthdiagram}) as the
$K^\vee$-homology of a commutative diagram of $K(1)$-local ${\cal
  E}_\infty$-ring spectra.  The authors have not been able to locate a
complete proof in the literature for the following result, though it
is known to the experts and a proof sketch can be found in
\cite[Remark 2.2]{davis-mahowald-connective-versions}.

\begin{prop}\label{prop:TMFtoKO} 
There is an ${\cal E}_\infty$-map $c_{K(1)}\co \TMF\to KO$ such that
$K^\vee_0(c_{K(1)})=a$ as in diagram~(\ref{diag:psthdiagram}).
\end{prop}

\begin{proof}
We remind the reader of the presentation of $KO$ and $TMF$ as finite
cell $L_{K(1)}\mb S$-algebras.  We will write $\mb PX$ for the free
$K(1)$-local $E_\infty$-ring spectrum on a $K(1)$-local spectrum $X$.

There is a generator $\zeta\in\pi_{-1}(L_{K(1)}S^0)$, and we define
$T_\zeta$ to be $\mb S \cup_\zeta e^0$: the pushout of the diagram
\[
\mb S \overfrom^{0} \mb PS^{-1} \overto^{\zeta} \mb S
\]
in the category of $K(1)$-local ${\cal E}_\infty$-ring spectra.
We refer to this as the ${\cal E}_\infty$-cone over $\zeta$.

There are elements $y$ and $f$ in $\pi_0 T_\zeta$. We refer
the reader to the discussion surrounding \cite[Proposition 5]{laures}
for the definitions of these elements, and to \cite[end of
appendix]{laures} for the existence of a factorization of the
attaching maps
\[
\xymatrix{
y\co \mb PS^0 \ar[rr]^{\theta(x)-h(x)} &&
 \mb PS^0 \ar[r]^{f} & T_\zeta.
}
\]
The spectrum $KO$ is $T_\zeta\cup_f e^1$, the ${\cal
  E}_\infty$-cone on $f$ \cite[Proposition 13]{hopkins}, and
$\TMF\simeq T_\zeta\cup_y e^1$ \cite[Convention on page 390]{laures}
as ${\cal E}_\infty$-ring spectra.  Therefore, there is an ${\cal
  E}_\infty$-map $c_{K(1)}\co \TMF\to KO$ factoring the given
attaching maps.

It remains to see that $K^\vee_0(c_{K(1)})=a$ and to this end, we
first consider the effect of $c_{K(1)}$ in homotopy. Remembering that
everything is implicitly $2$-completed, we know that
\begin{equation}\label{eq:homotopyofk1localtmf} \pi_0\TMF =
    \mathbb{Z}_2[f]= \mb Z_2[j^{-1}]
\end{equation}
by \cite[Proposition 6 and Lemma 9]{laures}.

By construction $f$ maps to zero under $c_{K(1)}$, and this
implies that $j^{-1}$ also maps to zero by the following
computation.

The element $j^{-1}$ is a $2$-adically convergent power series in $f$:
\begin{equation}\label{eq:expansion}
 j^{-1}=\sum\limits_{n=0}^\infty a_n f^n
\end{equation}
Clearly $c_{K(1)}$ sends $j^{-1}$ to $a_0$.

We now pass to $q$-expansions. It is classical that
$j^{-1}(q)=q+O(q^2)$. Since $f$ is of the form $f=\psi(b)-b$ for a
suitable $2$-adic modular function $b$ and $\psi$ the Frobenius
operator \cite[Equation (33)]{laures}, we learn that the $q$-expansion
\[
f(q)=\psi(b)(q)-b(q)=b(q^2)-b(q)
\]
has constant term $0$.  Hence, taking $q$-expansions of
equation~(\ref{eq:expansion}) and setting  $q=0$ yields $0=a_0$, as
desired.

Finally, knowing that $c_{K(1)}$ sends $j^{-1}$ to $0$ implies that
$K_0^\vee (c_{K(1)})=a$, because $K_0^\vee (c_{K(1)})$ is the map induced 
on the Igusa towers (\cite[Definition 5.6]{lawsonnaumann}) by the map
$\pi_0(c_{K(1)})$.
\end{proof}

All other maps of $K(1)$-local ${\mathcal E}_\infty$-ring spectra we
require will be constructed by obstruction theory. (The reason
the map $c_{K(1)}$ cannot be thus constructed is that $K^\vee_0(KO)$
is not an induced $\mb Z_2^\times$-module
(Equation~\ref{eq:khomofko}), which is a
manifestation of the fact that $KO$ is not complex orientable.)

We recall, for a graded $\psth$-algebra $B_*$ over $(\comp{K}_p)_*$,
$\Omega^t B_*$ is the kernel of the map of augmented $\psth$-algebras
\[
B_* \otimes_{(\comp{K}_p)_*}(\comp{K}_p)^* S^t \to B_*.
\]

\begin{prop}\label{prop:mappings}
Let $p$ be a prime and suppose $X$ and $Y$ are $K(1)$-local
${\mathcal E}_\infty$-ring spectra with the following properties.
\begin{enumerate}[i)]
\item $K_0^\vee (X)$ and $K^\vee _0(Y)$ are $p$-adically complete and
  $K^\vee _1(X)=K^\vee _1(Y)=0$.
\item The inclusion $\left(K^\vee _0(X)\right)^{\mb
    Z_p^\times}\subseteq K^\vee _0(X)$ is the $p$-adic completion of
  an ind-\'etale extension.
\item The ring $\left( K^\vee _0(X)\right)^{\mb Z_p^\times}$ is the
  $p$-adic completion of a smooth $\mb Z_p$-algebra.
\item For all $s>0$ we have $H^s_c(\mb Z_p^\times, K_*^\vee(Y))=0.$
\end{enumerate}
Then there is an isomorphism
\[
f \mapsto K^\vee_0(f)\co \pi_0 {\cal E}_\infty(X,Y)\stackrel{\cong}{\to}
Hom_{\psth}(K_0^\vee X, K_0^\vee Y),
\]
given by the canonical map evaluating on $K_0^\vee$, from the set
of connected components of the derived ${\mathcal E}_\infty$-mapping
space to the set of $\psth$-algebra maps.
\end{prop}

\begin{proof} The first assumption implies that $A_*:=K^\vee_*(X)$ and
  $B_*:=K_*^\vee (Y)$ are graded, $p$-adic, even-periodic
  $\psth$-algebras.  The remaining conditions are exactly those
  of \cite[Lemma 5.14]{lawsonnaumann}, application of which implies
  that for all $s\ge 2$ or $t\in\mb Z$ odd we have vanishing of
  the $\psth$-algebra cohomology groups
\[
H^s_{\psth}(A_*/(\comp{K}_p)_*,\Omega^t B_*).
\]
The claim now follows from Goerss-Hopkins obstruction theory as
  in \cite[Theorem 5.13, 3]{lawsonnaumann}.
\end{proof}

We will make use of the following particular 
instances of this result.

\begin{prop}\label{prop:mappingspaces}
For each dotted arrow between $K(1)$-local ${\cal
  E}_\infty$-ring spectra $X$ and $Y$ in the diagram
\[
\xymatrix{
&
\TMF \ar@{.>}[d] \ar@{.>}[dr] \ar@{.>}[dl] &
KO \ar@{.>}[d]\\
L_{K(1)} L_{K(2)} \tmf_1(3) &
\TMF_1(3) \ar@{.>}[r] &
K,
}
\]
there is an isomorphism
\[
f \mapsto K^\vee_0(f)\co \pi_0 {\cal E}_\infty(X,Y) \stackrel{\cong}{\to}
Hom_{\psth}(K_0^\vee X, K_0^\vee Y)
\]
given by evaluation on $K_0^\vee$.
\end{prop}

\begin{proof}
In order to deduce this from Proposition \ref{prop:mappings}, we
need to know certain properties of the $K^\vee$-homology of the
spectra involved, the local references for which we summarize in
the following table.

\begin{tabular*}{1.0\linewidth}{lccccc}
&$KO$ & $K$ & $\TMF_1(3)$ & $L_{K(1)} L_{K(2)}\TMF_1(3)$ & $\TMF$ \\
\text{$p$-adic, even}&(\ref{eq:structure-K-KO}) & (\ref{eq:structure-K-theory}) & \cite[5.4]{lawsonnaumann} & \cite[5.4]{lawsonnaumann} & (\ref{eq:khomoftmf})+(\ref{eq:structure-K-KO})\\
\text{ind-\'etale $K^\vee_0$}&(\ref{eq:structure-K-KO}) & (\ref{eq:structure-K-theory}) & \cite[5.8]{lawsonnaumann} & \cite[5.8]{lawsonnaumann} & (\ref{eq:khomoftmf})\\
\text{smooth subring}&(\ref{eq:structure-K-KO}) & (\ref{eq:structure-K-theory}) & \cite[6.1]{lawsonnaumann} & \cite[3.5]{lawsonnaumann} & (\ref{eq:homotopyofk1localtmf}) \\
\text{no cohomology}&* & (\ref{eq:structure-K-theory}) & \cite[5.8,3]{lawsonnaumann} & \cite[5.8,3]{lawsonnaumann} &*
\end{tabular*}

Here, the rows correspond to the itemized conditions in Proposition
\ref{prop:mappings} and the columns to the spectra under
consideration.  Note an entry means that the given spectrum satisfies
any assertions about either the domain spectrum $X$ or the target $Y$.
The statements labeled with an asterisk are actually false:
$K^\vee_0(KO)$ and $K^\vee _0(\TMF)$ are not cohomologically trivial
$\mb Z_2^\times$-modules.  These statements are not needed,
because this proposition does not make any assertions about maps {\em
  into} $KO$ or $\TMF$.
\end{proof}

\begin{cor}\label{cor:k1-diagram}
There is a diagram of ${\cal E}_\infty$-maps
\begin{equation}\label{eq:K1-diag}    
\xymatrix{
\TMF=L_{K(1)}\tmf  \ar[r]^-{c_{K(1)}} \ar[d]_{o_{K(1)}} & KO=L_{K(1)}ko\ar[d]^{\iota_{K(1)}} \\
\TMF_1(3)=L_{K(1)}\tmf_1(3) \ar[r]^-{\tilde{c}_{K(1)}} & K=L_{K(1)}ku,
}
\end{equation}
with diagram (\ref{diag:psthdiagram}) being realized by the
$K_0^\vee$-homology of diagram (\ref{eq:K1-diag}) and diagram
(\ref{eq:K1-diag}) commuting up to homotopy in the category of
  ${\cal E}_\infty$-ring spectra.
\end{cor}

\begin{proof}
Applying Proposition \ref{prop:mappingspaces}, we obtain
${\mathcal E}_\infty$-maps $o_{K(1)}$, $\tilde{c}_{K(1)}$, and
$\iota_{K(1)}$ that are characterized up to homotopy by
satisfying $K_0^\vee (o_{K(1)})=d$, $K_0^\vee (\tilde{c}_{K(1)})=c$, and
$K_0^\vee (\iota_{K(1)})=b$.  From Proposition \ref{prop:TMFtoKO} we
already have the ${\mathcal E_\infty}$-map $c_{K(1)}$ satisfying
$K_0(c_{K(1)})=a$.  Note that we do not need to know
whether $a$ is characterized by its effect in $K^\vee$-homology.

Proposition \ref{prop:mappingspaces} then reduces the homotopy
commutativity of diagram (\ref{eq:K1-diag}) to the previously
established commutativity of diagram (\ref{diag:psthdiagram}).
\end{proof}

\subsection{Chromatic gluing of maps}\label{subsec:che-glue-maps}

We briefly remind the reader of chromatic pullbacks in stable
homotopy, referring to the introduction of
\cite{goerss-henn-mahowald-rezk} for more details and references.

Fixing a prime $p$, every $p$-local spectrum $X$ maps canonically
to a tower of Bousfield localizations
\[
X\longrightarrow\left( \cdots L_n X\longrightarrow
L_{n-1}X\longrightarrow\cdots \longrightarrow
L_0X=X\otimes\mathbb{Q}\right),
\]
and the various stages of this tower are determined by canonical
homotopy pullbacks, called chromatic fracture squares
\cite{hovey-strickland-localisation} or \cite[Lecture 23, Proposition 5]{luriechromatic}:
\begin{equation}\label{diag:lntolnplus1}
\xymatrix{ L_n X\ar[r]\ar[d] & L_{K(n)}X\ar[d]\\
L_{n-1}X\ar[r] & L_{n-1}(L_{K(n)}X)}
\end{equation}
Here, $K(n)$ denotes any Morava $K$-theory of height $n$ at the prime
$p$ and the localization functor $L_n$ is naturally equivalent to
$L_{K(0) \vee \cdots \vee K(n)}$.  We will write $L_{K(1)} L_{K(2)} X$
for the iterated localization $L_{K(1)} L_{K(2)} X$, and similarly for
other iterates.

We will use other canonical homotopy pullbacks similar to
(\ref{diag:lntolnplus1}), such as the following:
\begin{equation}\label{diag:k1andk2glueing}
\xymatrix{ L_{K(1)\vee K(2)} Y\ar[r]\ar[d] & L_{K(2)}Y\ar[d]\\
L_{K(1)}Y\ar[r] & L_{K(1)} L_{K(2)}Y}
\end{equation}

\begin{lem}\label{lem:chrom-glue-maps}
  Assume $f_{K(i)}\co L_{K(i)}X\to L_{K(i)}Y$ ($i=1,2$) are ${\cal
    E}_\infty$-maps such that the diagram
\[
\xymatrix{
L_{K(1)}X \ar[r] \ar[d]^{f_{K(1)}} & L_{K(1)} L_{K(2)}X\ar[d]^{L_{K(1)}(f_{K(2)})} \\
L_{K(1)} Y\ar[r] & L_{K(1)} L_{K(2)} Y
}
\]
commutes up to homotopy in the category of ${\cal E}_\infty$-ring spectra.

Then there is an ${\cal E}_\infty$-map $f\co L_{K(1)\vee K(2)} X\to
L_{K(1)\vee K(2)}Y$, not necessarily unique, such that the diagrams
\[
\xymatrix{
L_{K(1)\vee K(2)} X\ar[r]^{f} \ar[d] & L_{K(1)\vee K(2)} Y\ar[d]
& & L_{K(1)\vee K(2)} X\ar[r]^{f} \ar[d] & L_{K(1)\vee
K(2)} Y\ar[d] \\
L_{K(1)}X\ar[r]^{f_{K(1)}} & L_{K(1)} Y & & L_{K(2)}X
\ar[r]^{f_{K(2)}} & L_{K(2)} Y
}
\]
commute up to homotopy.
\end{lem}

\begin{proof}
Apply the derived mapping-space functor ${\cal E}_\infty(L_{K(1)\vee
K(2)} X, -)$ to the chromatic fracture square
(\ref{diag:k1andk2glueing}).
\end{proof}

\begin{cor}\label{cor:o-k1-k2-local}
There exists an ${\mathcal E}_\infty$-map
\[ o_{K(1)\vee K(2)}\co L_{K(1)\vee K(2)}\tmf\to L_{K(1)\vee K(2)}\tmf_1(3)\]
such that the diagrams of ${\mathcal E}_\infty$-maps
\[
\xymatrix{
L_{K(1)\vee K(2)} \tmf\ar[r]_-{o_{K(1)\vee K(2)}} \ar[d] & L_{K(1)\vee
  K(2)} \tmf_1(3)\ar[d] \\
L_{K(1)}\tmf\ar[r]_-{o_{K(1)}} & L_{K(1)} \tmf_1(3)}
\hspace{1pc}
\xymatrix{
L_{K(1)\vee K(2)} \tmf\ar[r]_-{o_{K(1)\vee K(2)}} \ar[d] & L_{K(1)\vee
K(2)} \tmf_1(3)\ar[d] \\
L_{K(2)}\tmf \ar[r]_-{o_{K(2)}} & L_{K(2)} \tmf_1(3)
}
\]
both commute up to homotopy.
\end{cor}

\begin{proof}
This follows from Lemma \ref{lem:chrom-glue-maps}, provided we can
establish the commutativity up to homotopy of the following diagram of
${\mathcal E}_\infty$-ring spectra:
\begin{equation}\label{diag:k1}
\xymatrix{ L_{K(1)}\tmf\ar[rr]^-{o_{K(1)}} \ar[d]&  & 
L_{K(1)}\tmf_1(3)\ar[d]\\
L_{K(1)} L_{K(2)}\tmf\ar[rr]^-{L_{K(1)}(o_{K(2)})} & & 
L_{K(1)} L_{K(2)} \tmf_1(3)}
\end{equation}
The initial and terminal objects in this diagram appear in
  Proposition \ref{prop:mappingspaces}, and so it suffices to see
that the induced diagram in $K^\vee$-homology commutes:
\[
\xymatrix{V\ar[r]\ar[d] & K_0^\vee L_{K(1)}\tmf_1(3)\ar[d]\\
K_0^\vee L_{K(1)} L_{K(2)}\tmf\ar[r] & K_0^\vee L_{K(1)} L_{K(2)} \tmf_1(3)}
\]
This holds true because both composites classify isomorphic
trivializations of the elliptic curves $y^2 + xy + a^{-3}y = x^3$
and $y^2 + axy + y = x^3$ over $\comp{\mb Z\laur{a}}_2$
$\cong\pi_0L_{K(1)} L_{K(2)}\tmf_1(3)$ (see
Section~\ref{sec:elliptic-curve}).
\end{proof}


\begin{prop}\label{prop:k1-k2-diag}
There exists a diagram of ${\cal E}_\infty$-ring spectra which
commutes up to homotopy as follows:
\[
\xymatrix{
L_{K(1)\vee K(2)}\tmf \ar[rr]^-{c_{K(1)\vee K(2)}} \ar[d]^{o_{K(1)\vee
    K(2)}} & &
L_{K(1) \vee K(2)}ko\ar[d]^{\iota_{K(1)\vee K(2)}} \\
L_{K(1)\vee K(2)}\tmf_1(3)\ar[rr]^-{\tilde{c}_{K(1)\vee K(2)}} & &
L_{K(1) \vee K(2)}ku
}
\]
\end{prop}

\begin{proof}
We have the following diagram of ${\cal E}_\infty$-ring spectra
which commutes up to homotopy:
\begin{equation}\label{eq:K1-K2-diag}
\xymatrix{
L_{K(1)\vee K(2)}\tmf \ar[r] \ar[d]^{o_{K(1)\vee
    K(2)}} &
L_{K(1)} \tmf \ar[r]^-{c_{K(1)}} \ar[d]^{o_{K(1)}}& 
L_{K(1)}ko\ar[d]^{\iota_{K(1)\vee K(2)}} \\
L_{K(1)\vee K(2)}\tmf_1(3)\ar[r] &
L_{K(1)}\tmf_1(3) \ar[r]^-{\tilde c_{K(1)}} &
L_{K(1)}ku}
\end{equation}
Here, the left square is from
  Corollary~\ref{cor:o-k1-k2-local}, and the right one is from
  Corollary~\ref{cor:k1-diagram}.  Since $L_{K(1) \vee K(2)} ko
\simeq L_{K(1)} ko$ and similarly for $ku$, we can define the upper
and lower horizontal composites to be $c_{K(1) \vee K(2)}$ and $\tilde
c_{K(1) \vee K(2)}$ respectively.
\end{proof}

\subsection{The rational maps}\label{subsec:K0-local-maps}

We first note the following about rational ${\cal
    E}_\infty$-ring spectra.
\begin{lem}\label{lem:rationalmaps}
Suppose $X$ and $Y$ are ${\cal E}_\infty$-ring spectra such that
$\pi_* X \otimes \mb Q$ is a free graded-commutative $\mb Q$-algebra
on generators in even nonnegative degrees, and $\pi_* Y$ is rational
with homotopy in nonnegative odd degrees.  Then the natural map
\[
\pi_0 {\cal E}_\infty(X,Y) \to \Hom_{\text{graded rings}}(X_*,Y_*)
\]
is bijective, and all path components of the derived mapping space
${\cal E}_\infty(X,Y)$ are simply connected.
\end{lem}

\begin{proof}
Since $Y$ is rational, the natural map
\[
{\cal E}_\infty(X \otimes \mb Q, Y) \to {\cal E}_\infty(X,Y)
\]
is a weak equivalence.  As $X \otimes \mb Q$ is equivalent to a free $H\mb
Q$-algebra on some family of cells $x_i\co S^{2n_i} \to X$,
evaluation on the generators gives a weak equivalence, natural in
  $Y$, of the form
\[
{\cal E}_\infty(X \otimes \mb Q, Y) \to \prod \Omega^{\infty+{2n_i}} Y.
\]
The result follows by considering $\pi_0$ and $\pi_1$ of the
right-hand side.
\end{proof}

\begin{thm}
\label{thm:rationalcube}
There exists a strictly commutative diagram in the category of
rational ${\cal E}_\infty$-ring spectra as follows:
\begin{equation}
\label{eq:rationalcube}
\xymatrix{
L_{K(0)}\tmf \ar@{.>}[rrr] \ar@{.>}[d] \ar[ddr]& & &
L_{K(0)}ko \ar@{.>}[d] \ar[ddr] \\
L_{K(0)}\tmf_1(3) \ar@{.>}[rrr] \ar[ddr] & & &
L_{K(0)}ku \ar[ddr] \\
&L_{K(0)} L_{K(1) \vee K(2)} \tmf\ar[rrr] \ar[d] & & &
L_{K(0)} L_{K(1) \vee K(2)} ko \ar[d] \\
&L_{K(0)} L_{K(1)\vee K(2)} \tmf_1(3) \ar[rrr] & & &
L_{K(0)} L_{K(1)\vee K(2)} ku
}
\end{equation}
In this diagram, the bottom square is the rationalization of the
diagram displayed in Proposition \ref{prop:k1-k2-diag} and the
diagonal maps are arithmetic attaching maps.
\end{thm}

\begin{proof}
To construct diagram~(\ref{eq:rationalcube}) we must first construct
the maps in the top square of the cube as to render the entire diagram
homotopy commutative.
Recall:
\begin{align*}
\pi_*\tmf\otimes\mb Q&\cong \mb Q[c_4,c_6] &\text{where } &|c_i| = 2i\\
\pi_* \tmf_1(3)\otimes \mb Q&\cong \mb Q[A,B] && |A| = 2, |B| = 6
\text{ \cite[proof of Theorem 1.1]{lawsonnaumann}}\\
\pi_*ku\otimes\mb Q&\cong \mb Q[\beta] && |\beta| = 2 \\
\pi_*ko\otimes\mb Q&\cong \mb Q[\beta^2] &&
\end{align*}
In nonnegative degrees, the diagonal maps in
diagram~(\ref{eq:rationalcube}) are given on homotopy groups by
extension of scalars from $\mb Q$ to $\mb Q_2$.

Evaluating the modular forms $c_4$ and $c_6$ on
\[
y^2 + Axy + By = x^3,
\]
the universal elliptic curve with $\Gamma_1(3)$-structure used to
construct $\tmf_1(3)$, we find that
\[
c_4\mapsto A^4-24\, AB\, ,\, c_6\mapsto -A^6+36\, A^3B-216\, B^2.
\]
Similarly, evaluating at the Tate curve $y^2 + \beta xy = x^3$, we
find that
\begin{equation}
  \label{eq:cusp-orientation}
A\mapsto\beta,\,  B\mapsto 0,
\end{equation}
and
\[
c_4\mapsto \beta^4,\,c_6\mapsto -\beta^6.
\]
These formulas make diagram~(\ref{eq:rationalcube})
commutative on homotopy groups.

The homotopy groups of the spectra in the upper square of
diagram~(\ref{eq:rationalcube}) form polynomial algebras, and all
spectra in the diagram have zero homotopy in positive odd degrees.
Lemma~\ref{lem:rationalmaps} thus implies that constructing the maps
in this diagram is equivalent to defining the maps on homotopy
groups, and that the homotopy-commutativity of each square
subdiagram is equivalent to the commutativity of the square on
homotopy groups.  This shows that the cubical diagram commutes in the 
homotopy category.

Finally, the obstruction to lifting a homotopy commutative
cubical diagram to an honestly commutative, homotopy equivalent cubical
diagram lies in $\pi_1 {\cal E}_\infty(L_{K(0)} \tmf, L_{K(0), K(1)
  \vee K(2)} ku)$, which is the zero group (again by Lemma
\ref{lem:rationalmaps}).
\end{proof}

\begin{cor}
There is a commutative square of ${\cal E}_\infty$-ring
spectra as follows:
\begin{equation}\label{eq:K012-diag}
\xymatrix{ L_{K(0)\vee K(1)\vee K(2)}\tmf\ar[rrr]^{c_{K(0)\vee K(1)\vee K(2)}}\ar[d]_{o_{K(0)\vee K(1)\vee K(2)}} & & & L_{K(0)\vee K(1)\vee K(2)} ko\ar[d]^{\iota_{K(0)\vee K(1)\vee K(2)}} \\
L_{K(0)\vee K(1)\vee K(2)} \tmf_1(3) \ar[rrr]^{\tilde{c}_{K(0)\vee K(1)\vee K(2)}} &  & & L_{K(0)\vee K(1)\vee K(2)} ku}
\end{equation}
\end{cor}

\begin{proof}
For the $K(0) \vee K(1) \vee K(2)$-local spectra under
consideration, $L_{K(1) \vee K(2)}$ is $p$-adic completion.  We have
canonical arithmetic squares
\[
\xymatrix{ L_{K(0)\vee K(1)\vee K(2)} Y\ar[rr]\ar[d] 
& & L_{K(1)\vee K(2)}Y\ar[d]\\
L_{K(0)}Y\ar[rr]
& & L_{K(0)} L_{K(1)\vee K(2)} Y.}
\]
We can then take levelwise homotopy pullbacks of the maps
\[
L_{K(0)}Y \to L_{K(0)} L_{K(1)\vee K(2)} Y  \leftarrow L_{K(1)\vee K(2)} Y
\]
from the diagonals of diagram~(\ref{eq:rationalcube}) and obtain
the desired commutative square.
\end{proof}

\section{The cohomology computation}\label{sec:cohomology}

The techniques used in this section are very similar to those employed
in \cite{rezk-tmfnotes} to calculate $H^*(\tmf)$.

Let $p$ be a prime, abbreviate $H:=H\mb F_p$ and recall that the
  dual Steenrod algebra $A_* = H_* H$ takes the form
\[
A_* \cong \begin{cases}
P\left(\bxi_1, \bxi_2, \ldots,\right)
&\text{if }p=2,\\
P(\bxi_1, \bxi_2, \ldots )\otimes E(\btau_0, \btau_1\ldots)
&\text{if }p \neq 2.
\end{cases}
\]
Suppose a $p$-local spectrum $X$ is connective and of finite
type, with a map $X \to H$, such that the mod-$p$ homology maps
  isomorphically to the sub-Hopf-algebra of $A_*$ given by
\[
H_* X \cong \begin{cases}
P\left(\bxi_1^2,\ldots,\bxi_{n+1}^2,\bxi_{n+2},\ldots\right)
&\text{if }p=2,\\
P(\bxi_1,\ldots )\otimes E(\overline{\tau}_{n+1},\ldots)
&\text{if }p \neq 2.
\end{cases}
\]
For example, this is true when $X = \BPP{n}$.  In these
circumstances the Adams spectral sequence degenerates, and we find
that
\[
\pi_*X\cong\mb Z_{(p)}[v_1,\ldots,v_n]
\]
where $|v_i|=2(p^i-1)$ \cite[Chapter 4, Section 2, page
111]{ravenel}.  (This is only necessarily an isomorphism as
graded abelian groups unless $X$ is a homotopy commutative and
associative ring spectrum.)
We will establish a converse to this isomorphism under the
assumption that $X$ is a ring spectrum in Theorem
\ref{thm:coh-comp}.
\begin{defn}\label{def:generalbpn}
A $p$-local ring spectrum $R$ is a {\em generalized $\BPP{n}$} if it
admits a complex orientation such that the resulting composite map
\[
\mb Z_{(p)}[v_1,\ldots,v_n]\subseteq \pi_*BP \to \pi_*MU_{(p)} \to\pi_* R
\]
is an isomorphism.
\end{defn}

We remark that as the element $v_i$ is an invariant of the formal
group modulo $(p,v_1,\ldots,v_{i-1})$, the property of a $p$-local,
homotopy commutative, complex orientable ring spectrum being a
generalized $\BPP{n}$ depends only on the ring structure and is
independent of the choice of complex orientation. In particular,
it does not depend on the choice of Hazewinkel, Araki, or arbitrary
other $p$-typical $v_i$-classes.

The following fact served as the basis for the construction of
$\tmf_1(3)$ by chromatic fracture.

\begin{prop}
If $n > 0$, any generalized $\BPP{n}$ is the connective cover of
its $L_n$-localization (or equivalently its $L_{K(0) \vee K(1) \vee
  \cdots \vee K(n)}$-localization).
\end{prop}

\begin{proof}
Let $L_n^f$ denote the finite localization of Miller
\cite{miller-finitelocalizations}.  The fiber of the map $L_n^f \to
L_n$ is $BP$-acyclic.  As both localizations are smashing, this is
also true for all (homotopy) $BP$-modules, including $p$-local complex
orientable ring spectra.

Therefore, the fiber of the localization map $\BPP{n} \to L_n \BPP{n}$
is equivalent to the finite colocalization $C^f \BPP{n}$.  By
\cite[Proposition~7.10(a)]{hovey-strickland-localisation}, this finite
colocalization is an appropriate homotopy colimit of function spectra
\[
\hocolim F(M(p^{i_0}, v_1^{i_1}, \ldots, v_n^{i_n}), \BPP{n})
\]
out of a tower of generalized Moore spectra.  The homotopy groups of
this function spectrum are
\[
\Sigma^{-(\sum i_k |v_k|)-n-1} \BPP{n}_*/(p^{i_0}, v_1^{i_1}, \ldots, v_n^{i_n}),
\]
and the colimit is ultimately
\[
\Sigma^{-n-1} \BPP{n}_*/(p^\infty, v_1^\infty, \ldots, v_n^\infty)
\]
whose top homotopy group is in degree $-n-1-\Sigma_{k=0}^n (2p^k-2)$.
As $n > 0$, the map $\BPP{n} \to L_n^f \BPP{n}$ is then a model
for the connective cover.
\end{proof}

There does not appear to be an easier method to prove this than direct
calculation.  There are many closely related spectra where the
connective cover is not the correct tool (such as those associated to 
moduli problems where the underlying curve has positive genus).

The authors are not aware of any results establishing uniqueness
of $\BPP{n}$ when $n \geq 2$.  It is not clear when, after
forgetting the ring spectrum structure, two generalized $\BPP{n}$ with
nonisomorphic formal groups might have the same underlying homotopy
type.  It is also not clear, for any particular formal group of the
correct form, how many weak equivalence classes of generalized
$\BPP{n}$ might exist realizing this formal group.  (There exist
results if one assumes additional structure, such as that of an
$MU$-module or $MU$-algebra; see, for example,
\cite{jeanneret-wuthrich-quadratic}.)

The following is a consequence of Definition \ref{def:generalbpn} and
of the existence of the Quillen idempotent splitting $MU_{(p)} \to
BP$.
\begin{lem}
\label{lem:ringmaps}
Suppose $R$ is a generalized $\BPP{n}$.  Then there are maps of
  ring spectra $BP \to R \to H$, with the former map $(2p^{n+1} -
2)$-connected and the latter unique up to homotopy.  If $R$ is an
${\cal A}_\infty$-ring spectrum or an ${\cal E}_\infty$-ring spectrum,
then the map $R \to H$ is a map of ${\cal A}_\infty$-ring spectra or
${\cal E}_\infty$-ring spectra accordingly.
\end{lem}

\begin{thm}
\label{thm:coh-comp}
Suppose $R$ is a generalized $\BPP{n}$.  Then the map $R\to H$
induces an isomorphism of $H^* R$ with the left ${\cal A}^*$-module
${\cal A}^*/\!/E(n)$, and of $H_* R$ with the subalgebra
$B_*$ of the dual Steenrod algebra given as follows:
\[
B_* =
\begin{cases}
P(\bxi_1^2, \ldots, \bxi_{n+1}^2, \bxi_{n+2}, \ldots) &\text{if }p=2,\\
P(\bxi_1^2, \ldots) \otimes E(\btau_{n+1}, \ldots) &\text{if }p\neq 2.
\end{cases}
\]
\end{thm}

\begin{proof}
We recall from \cite[Theorem~3.4]{baker-jeanneret} that the
Brown-Peterson spectrum $BP$ admits an ${\cal A}_\infty$ ring
structure (in fact, it admits many).  In the following we
choose one for definiteness. Smashing the maps from
Lemma~\ref{lem:ringmaps} on the right with $BP$ gives a sequence of
maps
\[
BP \smsh{} BP \to R \smsh{} BP \to H \smsh{} BP.
\]
Complex orientability of $BP$, $R$, and $H$ implies that on
homotopy groups, this becomes the sequence of maps of polynomial
algebras
\[
BP_*[t_i] \to R_*[t_i] \to \mb F_p[t_i],
\]
with the polynomial generators $t_i$ mapped identically.
We then apply the natural equivalence $(- \smsh{} BP) \smsh{BP} H
\simeq (-) \smsh{} H$, together with the K\"unneth spectral sequence
\cite[Theorem IV.4.1]{ekmm}, to obtain a natural map of spectral
sequences:
\begin{equation}
  \label{eq:kunneth-sseqs}
\Tor^{BP_*}_{**}(R_*[t_i], \mb F_p) \to
\Tor^{BP_*}_{**}(\mb F_p[t_i], \mb F_p),
\end{equation}
which strongly converges to the map $R_*H\to H_* H$.
Here the action of the generators $v_i \in BP_*$ is through their
images under the right unit $BP_* \stackrel{\eta_R}{\to} BP_*[t_i]$.
Modulo $(p,v_1,\ldots,v_{k-1})$, the image of $v_k$ under the right
unit is equal to $v_k$.  

Writing $S_* = BP_*/(p,v_1,\ldots,v_n)$, we have an identification
of derived tensor products
\[
(- \otimes^{\mb L}_{BP_*} S_*) \otimes^{\mb L}_{S_*} \mb F_p \cong (-)
\otimes^{\mb L}_{BP_*} \mb F_p.
\]
This shows that the map of equation~(\ref{eq:kunneth-sseqs}) is the
abutment of a map of Cartan-Eilenberg spectral sequences:
\begin{equation}
  \label{eq:cartan-sseqs}
\Tor^{S_*}_{**}(\Tor^{BP_*}_{**}(R_*[t_i], S_*), \mb F_p) \to 
\Tor^{S_*}_{**}(\Tor^{BP_*}_{**}(\mb F_p[t_i], S_*), \mb F_p)
\end{equation}
The elements $(p,v_1,\ldots,v_n)$ form a regular sequence in
$R_*$.  The image of the regular sequence $(p,v_1,\ldots ,v_n)\in
BP_*$ under the map
\[ BP_*\stackrel{\eta_R}{\to} BP_*BP\to R_*BP\simeq R_*[t_i]\]
is therefore a regular sequence, by induction, because
every $v_k$ is invariant modulo $(p,\ldots,v_{k-1})$ and because
  of the assumed properties of the map $BP_*\to R_*$.

This shows that the higher Tor-groups in
\[
\Tor^{BP_*}_{**}(R_*[t_i], S_*)
\]
are zero, with the zero'th term given by the tensor product $R_*[t_i]
\otimes_{BP_*} S_* \cong \mb F_p[t_i]$.

By contrast, the image of $v_k$ in $\mb F_p[t_i]$ under the right unit
is zero for all $k$, and hence the Tor-algebras are exterior
algebras.

Therefore, the map of equation~(\ref{eq:cartan-sseqs}) degenerates to
an edge inclusion:
\[
\mb F_p[t_i] \otimes \Lambda[x_{n+1}, x_{n+2}, \ldots] \to \mb
F_p[t_i] \otimes \Lambda[x_1, \ldots, x_n] \otimes \Lambda[x_{n+1}, x_{n+2},
\ldots]
\]

Here $t_i$ is in total degree $2p^i - 2$ and $x_i$ is in total degree
$2p^i - 1$.  For the right-hand term, the associated graded vector
space already has the same dimension in each total degree as the dual
Steenrod algebra.  Therefore, both the Cartan-Eilenberg and K\"unneth
spectral sequences must degenerate, as the final target is the dual
Steenrod algebra and any non-trivial differentials would result in a
graded vector space with strictly smaller dimension in some degree.

We find that the map $R_* H \to H_* H$ is an inclusion of right
comodules over the dual Steenrod algebra, and the image below degree
$2p^{n+1}-1$ consists only of terms in even degrees.  On cohomology,
this implies that the map ${\cal A}^* \to H^* R$ is a surjection of
left ${\cal A}^*$-modules, and the image of the generator $1 \in {\cal
A}^*$ is acted on trivially by the odd-degree Milnor primitives
$Q^0,\ldots, Q^n$.  The induced map ${\cal A}^*/\!/E(n) \to H^* R$ is
still a surjection and both sides are graded vector spaces of the
same, levelwise finite, dimensions over $\mb F_p$.

This shows that $H^* R$ has the desired form.  The statement for
homology follows by dualizing the cohomology description.
\end{proof}

\section{Proof of Theorem \ref{thm:main}}\label{sec:proof}

We have now assembled all the preliminaries needed to give the proof
of Theorem~\ref{thm:main}, For ease of reference, we recall the
statements we need to prove.

\begin{enumerate}[i)]
\item There is a commutative diagram of connective ${\cal E}_\infty$-ring
    spectra as follows:
\[
\xymatrix{ \tmf_{(2)} \ar[r]^{c}\ar[d]^{o} & ko_{(2)}\ar[d]^{\iota} \\
\tmf_1(3)_{(2)} \ar[r]^{\tilde{c}} & ku_{(2)}}
\]
\item In mod-$2$ cohomology, this induces the
following canonical diagram of modules over the mod $2$ Steenrod
algebra ${\cal A}^*$:
\[ \xymatrix{ {\cal A}^*/\!/{\cal A}(2) & {\cal A}^*/\!/{\cal A}(1)\ar[l]\\
{\cal A}^*/\!/E(2)\ar[u]& {\cal A}^*/\!/E(1).\ar[l]\ar[u]}\]
\item There exists a complex orientation of $\tmf_1(3)_{(2)}$ such that in
homotopy $\tilde{c}$ sends the Hazewinkel generators $v_1$ to $v_1$
and $v_2$ to zero.
\item There is a cofiber sequence of $\tmf_1(3)_{(2)}$-modules
\[
\Sigma^6 \tmf_1(3)_{(2)}\stackrel{\cdot v_2}{\longrightarrow} \tmf_1(3)_{(2)}
\stackrel{\tilde{c}}{\longrightarrow} ku_{(2)}.
\]
\end{enumerate}

\begin{proof}
The existence of the desired commutative diagram of ${\mathcal
E}_\infty$-ring spectra is established by taking connective
covers of diagram~(\ref{eq:K012-diag}).

It is well known that there are isomorphisms of ${\cal A}^*$-modules
$H^*(ko)\cong {\cal A}^*/\!/{\cal A}(1)$ and $H^*(ku)\cong {\cal
  A}^*/\!/E(1)$.  Theorem \ref{thm:reminder}, together with Theorem
\ref{thm:coh-comp}, implies $H_*\tmf_1(3)\cong
P(\bxi_1^2,\bxi_2^2,\bxi_3^2, \bxi_4,\ldots)$ and
$H^*\tmf_1(3)\cong {\cal A}^*/\!/E(2)$.  A well-known result, based on
work on $\tmf$ initiated by Hopkins, Mahowald and Miller, is that
$H^*(\tmf)\cong {\cal A}/\!/{\cal A}(2)$ as a module over the Steenrod
algebra \cite[Theorem 9.2]{hopkins-mahowald}. To the best of the
authors' knowledge, this result still awaits official documentation.
A sketch based on the characterization of \cite[Theorem
14.5]{rezk-tmfnotes} can be found in \cite[Section 21]{rezk-tmfnotes}.

The diagram of ${\cal A}^*$-modules in the statement of the theorem
commutes because all appearing ${\cal A}^*$-modules are cyclic,
generated by $1$.

To address the remaining statements, note that the map
$\pi_*(\tilde{c})\co\pi_*\tmf_1(3)_{(2)} \to\pi_*ku_{(2)}$ is determined by
  its rationalization
\[
\mb Q [A,B] \to \mb Q [\beta],
\]
where we have $A\mapsto\beta$ and $B\mapsto 0$ by construction (see
Theorem~\ref{thm:rationalcube}).  Now recall that there exists an
orientation $BP\to \tmf_1(3)_{(2)}$ which maps $v_1$ to $A$ and $v_2$
to $B$ by \cite[Proposition 8.2]{lawsonnaumann}.

The composite $\tmf_1(3)_{(2)}$-module map $\tilde{c}\circ(\cdot
v_2)\co\Sigma^6 \tmf_1(3)_{(2)} \to ku_{(2)}$ sends the generator to
$0=\pi_*(\tilde{c})(v_2)\in\pi_6(ku_{(2)})$, and hence there is a
factorization in the category of $\tmf_1(3)_{(2)}$-modules
\[
\xymatrix{
\Sigma^6\tmf_1(3)_{(2)} \ar[rr]^{\cdot v_2} \ar[drr]^{0} & & \tmf_1(3)_{(2)}
\ar[rr]\ar[d]^{\tilde c} & & \mathrm{cof}((\cdot v_2))\ar@{.>}[dll]\\
&&
ku_{(2)}
}
\]
Examining homotopy groups, we get an induced equivalence between
$ku_{(2)}$ and the cofiber of $v_2$ as spectra, and hence as
$\tmf_1(3)_{(2)}$-modules.
\end{proof}

\appendix
\section{Appendix: Forms of $K$-theory}\label{app:formsofKU}

In this section we describe how to functorially construct certain forms of
$K$-theory \cite{morava-forms} as ${\cal E}_\infty$-ring spectra, and then
give a discussion of a form of $K$-theory related to $\tmf_1(3)_{(2)}$.
The core content is a restatement of the fact that complex conjugation
acts on $KU$ by ${\cal E}_\infty$-ring maps, that the group $\mb
Z_p^\times$ acts on $\comp{KU}_p$ by ${\cal E}_\infty$-ring maps,
and that the element $-1$ acts compatibly with complex conjugation.
(However, these forms receive less attention than they
might, and the reader who has not read Morava's paper recently
deserves a reminder to do so.)  

Indeed, some of this section could be regarded as consequences of the
Goerss-Hopkins-Miller theorem
\cite{rezk-hopkinsmiller,goerss-hopkins}. After $p$-completion, any
form of $K$-theory that we construct is a Lubin-Tate spectrum for the
formal group law over its residue field, and the spaces of ${\cal
E}_\infty$-maps are homotopically discrete and equivalent to certain
sets of isomorphisms between the associated formal group laws.

\begin{defn}
A {\em form of the multiplicative group scheme} over $X$ is a
group scheme over $X$ which becomes isomorphic to $\mb G_m$ after a
faithfully flat extension.  (This is the same as a one-dimensional
torus in the sense of SGA3.)

If $X$ is a formal scheme over $\Spf(\mb Z_p)$, a {\em form of the
formal multiplicative group} over $X$ is a $1$-dimensional formal group
over $X$ whose reduction to $X/p$ is of height $1$.
\end{defn}

\begin{rmk}
There are numerous examples of forms of $\mb G_m$.  For $b$ and $c$ in
$R$, there is a group scheme structure with unit $0$ on the
complement of the roots of $1 + bx + cx^2$ in $\mb P^1$, given by
\[
F(x,y) = \frac{x+y+bxy}{1 - cxy}.
\]
The isomorphism class of such an object (by an isomorphism fixing the
invariant diffential) is determined by the isomorphism class of $y^2 +
by + c$ under translations $y \mapsto y + r$, and the object is a form
of $\mb G_m$ if and only if the discriminant $b^2 - 4c$ is invertible.
Morava was originally interested in the specific formal group laws of
the form
\[
F(x,y) = \frac{x+y+(1-a)xy}{1 + axy}.
\]
\end{rmk}

\begin{defn}
We denote by $\multmod / \Spec(\mb{Z})$ the stack which is is
the moduli of forms of the multiplicative group scheme $\mb G_m$.
For a fixed prime $p$, we denote by $\fmultmod / \Spf(\mb Z_p)$ the
stack which is the moduli of forms of the formal multiplicative group
$\widehat{\mb G}_m$.
\end{defn}

\begin{prop}
The stack $\multmod$ is equivalent to the stack $BC_2$ classifying
principal $C_2$-torsors, and the stack $\fmultmod$ is equivalent to
the stack $B\mb Z_p^\times = \lim B(\mb Z/p^k)^\times$ classifying
compatible systems of principal $(\mb Z/p^k)^\times$-torsors.
\end{prop}

\begin{proof}
The group scheme $\mb G_m$ is defined over $\mb Z$ and its sheaf of
automorphisms is the constant group scheme $C_2$ of order two; as a
result, for any $X$ equipped with $\mb G \to X$ a form of $\mb G_m$,
there is a principal $C_2$-torsor $Y = {\rm Iso}_{X}(\mb G_m, \mb G)$
classifying isomorphisms between $\mb G$ and $\mb G_m$; on $Y$ there
is a chosen isomorphism $\mb G_m \to \mb G$.  Conversely, given such a
$C_2$-torsor $Y \to X$ we can recover $\mb G$ as the quotient $\mb G_m
\times_{C_2} Y \to \Spec(Z) \times_{C_2} Y = X$.

In other language, the moduli stack $\multmod$ is equivalent to the
stack $BC_2$ classifying principal $C_2$-torsors.

Similarly, for a formal group law $\mb G$ let $\mb G[p^k]$ be the
group scheme of $p^k$'th roots of unity.  The sheaf of automorphisms
of $\mb G_m[p^k]$ is the constant group scheme $(\mb Z/p^k)^\times$.
Any formal group $\mb G$ of height $1$ on a formal scheme $X$ over
$\mb Z_p$ carries a sequence of covers $Y_k = {\rm Iso}(\mb G_m[p^k], \mb
G[p^k])$; any height $1$ formal group \'etale-locally has its torsion
isomorphic to that of $\mb G_m$, and so the $Y_k$ form a tower of
principal $(\mb Z/p^k)^\times$-torsors on $X$.  Conversely, the
inductive system of schemes $\mb G[p^k]$ can be recovered as $\mb
G_m[p^k] \times_{(\mb Z/p^k)^\times} Y_k$.

Because $X$ is a formal scheme over $\Spf(\mb Z_p)$, $X$ is a colimit
of schemes where $p$ is nilpotent; we then wish to know that the
formal group $\mb G$ is equivalent data to the directed system $\mb
G[p^k]$.  One can recover this from an equivalence between formal
groups of height $1$ and $p$-divisible groups of height $1$ and
dimension $1$.

At its core, however, this is locally based on the observation that
for a formal group {\em law} of height $1$ on a $p$-adically complete
ring $R$, the Weierstrass preparation theorem implies that the Hopf
algebra $R\pow{x}/[p^k](x)$ representing $\mb G[p^k]$ is free on the
basis $\{1,x,\ldots,x^{p^k-1}\}$.  Both $R\pow{x}$ and its
multiplication are recovered uniquely by the inverse limit of these
finite stages.  The map $\mb G[p] \to \mb G[p^k]$ induces an
isomorphism on cotangent spaces at the identity, and so any coordinate
on $\mb G[p]$ automatically produces coordinates on $\mb G[p^k]$ and
then $\mb G$ itself.
\end{proof}

Both of these moduli stacks naturally carry $1$-dimensional formal
groups, as follows. The stack $\fmultmod$ carries a universal formal
group by definition, and the stack $\multmod$ by taking the completion
of its universal group scheme (which is affine and one-dimensional
since these are flat-local properties). Therefore, it makes
sense to ask if these formal groups can be realized by spectra;
see \cite[Section 4.1]{goerss-bourbaki} for details on such
realization problems.  The formal groups give rise to a
commutative diagram of maps to the moduli stack ${\cal M}_{fg}$ of
formal groups:
\[
\xymatrix{
\comp{(\multmod)}_p \ar[d] \ar[r] &
\fmultmod \ar[d] \\
\multmod \ar[r] &
{\cal M}_{fg}
}
\]
Here $\comp{(\multmod)}_p$ is the base change of $\multmod$ to
$\Spf(\mb Z_p)$.  The realization problem for most of the above
diagram has a solution.
\begin{thm}\label{realizeKU}
There exist lifts ${\cal O}^{top}_{\multmod}$
(resp. ${\cal O}^{top}_{\fmultmod}$) of the structure sheaves of
$\multmod$ (resp. $\fmultmod$) to sheaves of weakly even-periodic
$E_\infty$-ring spectra, along with an isomorphism between the
formal group from the complex-orientable spectrum structure 
and the formal group pulled back from ${\cal M}_{fg}$.  The homotopy
groups of both ${\cal O}^{top}$ and ${\cal O}$ in degree $2k$ are, on
affine charts, the tensor powers $\omega^{\otimes k}$ of the sheaf of
invariant differentials.
\end{thm}

Given a diagram of solutions to the realization problem on the
above stacks, one would expect to obtain the following upon taking global
sections.
\[
\xymatrix{
\comp{KO}_p & \mb S_{K(1)} \ar[l] \\
KO \ar[u] & \mb S \ar[l] \ar[u]
}
\]
The maps in this diagram are all unit maps for the associated ring
spectra.

\begin{proof}
We first construct pre-sheaves of ${\cal E}_\infty$-ring spectra defined on 
affines.

Given an \'etale map $\Spec(R) \to \multmod$, we form the pullback:
\[
\xymatrix{
\Spec(T) \ar[r] \ar[d] &
\Spec(R) \ar[d] \\
\Spec(\mb Z) \ar[r] &
\multmod
}
\]
(Since the constant group scheme $\mb Z/2\mb Z$ is affine, so is
the canonical map $\Spec(\mb Z)\to \multmod$, showing the above
pull-back is indeed affine.)

Then $\Spec(T) \to \Spec(R)$ is a Galois cover with Galois group
$C_2$. The cohomology groups $H^i(C_2; T)$ are therefore trivial for
$i > 0$, and the fixed subring $H^0(C_2,T)$ is equal to $R$.
Furthermore, if $\epsilon$ denotes the $C_2$-module $\mb Z$ with the
sign action, the cohomology groups $H^i(C_2; \epsilon \otimes
T)$ vanish for $i>0$.  (This will be relevant because
$\pi_2KU\cong\epsilon$ as a $C_2$-module.)

The map $\mb Z \to T$ is \'etale, and so there is a homotopically
unique, and $C_2$-equivariant, realization $\mb S \to \mb S(T)$
such that $\pi_* \mb S(T) = \pi_* \mb S \otimes T$
by the results of \cite[Section 2]{baker-richter-algebraicgalois}
  or \cite[Theorem 7.5.0.6]{higheralgebra}. Write
$\sigma$ for the generator of $C_2$.  The $C_2$-action on $T$ is
compatible with the negation action on the multiplicative group scheme
over $T$.

Since $\pi_*\mb S(T)=\pi_*\mb S\otimes_{\mb Z} T$ is flat over
$\pi_*\mb S$, we have an isomorphism
\[
\pi_*(KU\smsh{}\mb S(T))\cong \pi_*KU\otimes_{\pi_*\mb S}\pi_*\mb
S(T)\cong \pi_*KU\otimes_{\mb Z} T.
\]
The group $C_2$ acts via the diagonal $\psi^{-1} \smsh{} \sigma$  on
$KU\smsh{} \mb S(T)$; this action lifts the $C_2$-action on the
formal group law of $KU \smsh{} \mb S(T)$.  We therefore obtain a
homotopy fixed-point spectrum ${\cal O}^{pre}_{\multmod}(R) := (KU \smsh{}
\mb S(T))^{hC_2}$. The vanishing of the higher group cohomology
implies that the homotopy fixed-point spectral sequence degenerates
into an isomorphism
\[
\pi_*{\cal O}^{pre}_{\multmod}(R)\stackrel{\cong}{\to} H^0(C_2 ;
KU_*\otimes_{\mb Z}T).
\]
We have
\begin{equation}
\pi_2{\cal O}^{pre}_{\multmod}(R)\cong
H^0(C_2,\epsilon\otimes T) \cong T^{\sigma=-1}\subseteq T.
\end{equation}
This is a projective $R$-module of rank $1$ and hence invertible,
and one concludes that this ring of invariants is weakly
even-periodic.  An application of \cite[Lemma 3.8]{lawsonnaumann}
shows the formal group of ${\cal O}^{pre}_{\multmod}(R)$ is the one
classified by the given map $R\to\multmod$.

The construction of $\mb S(T)$ can be made functorial in the
$C_2$-equivariant algebra $T$, and so this gives rise to the desired
presheaf ${\cal O}^{pre}_{\multmod}$ on affine objects over
$\multmod$.

We now discuss a similar setup for $\fmultmod$ and forms of the
multiplicative formal group, based on work of Behrens and Davis.

Following \cite[Section~8]{behrens-davis-fixedpoint}, let $F_1$
be the homotopy colimit of the Devinatz-Hopkins homotopy fixed-point
spectra:
\[
F_1:=\colim_{n} (\comp{KU}_p)^{dh(1 + p^n \mb Z_p)}
\]
This is a discrete ${\cal E}_\infty$ $\mb Z_p^\times$-spectrum
with $K(1)$-localization $\comp{KU}_p$.

Given a $p$-adic $\mb Z_p$-algebra $R$ with an
\'etale map
\[
\Spf(R) \to \fmultmod = \lim [\Spf(\mb Z_p) /\!/ (\mb Z/p^n)^\times]
\]
classifying a formal group $\mb G$ over $R$, we form
the system of pullbacks
\[
\xymatrix{
\Spf(T_n) \ar[rr] \ar[d] &&
\Spf(\mb Z_p) \ar[d] \\
\Spf(R) \ar[r] &
\fmultmod \ar[r] &
[\Spf(\mb Z_p) /\!/ (\mb Z/p^n)^\times]
}
\]
The maps $\Spf(T_n) \to \Spf(R)$ are the Galois covers with Galois
group $(\mb Z/p^n)^\times$ trivializing $\mb G[p^n]$.  We consider the 
$\mb Z_p$-algebra $T = \colim T_n$, which is an
ind-Galois extension of $R$ with Galois group $\mb
Z_p^\times$ and discrete $\mb Z_p^\times$-action.  The maps $\mb Z_p
\to T_n$ are \'etale because the classifying map of $\mb G$ is.

As in the case of $\multmod$ above, we functorially realize this
to obtain a $\mb Z_p^\times$-equivariant directed system $\comp{\mb
  S}_p \to \comp{\mb S}_p(T_n)$ of ${\cal E}_\infty$-ring spectra,
with the action on $\comp{\mb S}_p(T_n)$ factoring through $(\mb
Z/p^n)^\times$, such that $\pi_* \comp{\mb S}_p(T_n) = \pi_* \comp{\mb
  S}_p \otimes_{\mb Z_p} T_n$.  We finally introduce the discrete
${\cal E}_\infty$ $\mb Z_p^\times$-spectrum
\[
\comp{\mb S}_p(T)=\hocolim_n \comp{\mb S}_p(T_n).
\]
The homotopy of $\comp{\mb S}_p(T)$ is isomorphic to $\pi_*
  \comp{\mb S}_p \otimes_{\mb Z_p} T$.
The ${\cal E}_\infty$-ring spectrum $F_1 \smsh{} \comp{\mb
S}_p(T)$ is a discrete $\mb Z_p^\times$-spectrum with the diagonal
action, acting as the Morava stabilizer on $F_1$ and via the Galois
action on $\comp{\mb S}_p(T)$.  As this spectrum is $E(1)$-local,
the $K(1)$-localization of $F_1\smsh{}\comp{\mb S}_p(T)$ is equivalent
to the $p$-completion $\comp{(KU \smsh{} \comp{\mb
  S}_p(T))}_p$, with homotopy groups isomorphic to $\pi_* \comp{KU}_p
\widehat \otimes_{\mb Z_p} \comp{T}_p$.  We define 
  our presheaf to take $R$ to be a continuous homotopy
  fixed-point object:
\[
{\cal O}^{pre}_{\fmultmod}(R) = L_{K(1)} \left((F_1 \smsh{} \comp{\mb
    S}_p(T))^{h\mb Z_p^\times}\right)
\]
The homotopy fixed-point spectrum is $E(1)$-local, and so
  $K(1)$-localization is still simply $p$-adic completion.

We will now show that the resulting spectrum is even-periodic using the
homotopy fixed-point spectral sequence (see \cite[Theorem 5.1]{lawsonnaumann})
\[
H^s_c(\mb Z_p^\times, (\comp{KU}_p)_t\otimes_{\mb Z_p} \comp{T}_p)
\Rightarrow \pi_{t-s}({\cal O}^{pre}_{\fmultmod}(R)).
\]
Fix $n\in\mb Z$.  We will show that the continuous cohomology of $\mb
Z_p^\times$ with coefficients in
$W:=\comp{KU}_{p,2n}\hat{\otimes}_{\mb Z_p} \comp{T}_p$ vanishes for
$s > 0$, and the zero'th cohomology group is free of rank one over
$R$.  Multiplication by the Bott element shows that the $\mb
Z_p^\times$-module $W$ is isomorphic to $\comp{T}_p$ twisted by the
$p$-adic character $\alpha\mapsto \alpha^{n}$.

Since $\Spf(R) = \Spf(T_1) \to \Spf(\mb Z_p)$ is \'etale, $T_1 = R$
must be isomorphic to a finite product of copies of $W(\mb F_q)$ with
$q$ a power of $p$.  Therefore, without loss of generality we may
assume $R = W(\mb F_q)$.

For any $m \ge 1$, there exists a sufficiently large $k$ so that the
subgroup $U := (1 + p^k \mb Z_p)^\times < \mb Z_p^\times$ acts trivially on
$\mb Z/p^m$ by the character $\alpha \mapsto \alpha^n$.  The
continuous cohomology of $U$ with coefficients in $W/p^m$ then
coincides with the continuous cohomology of $U$ with coefficients in
$T/p^m$.  However, as $T/p^m$ is a Galois extension of $T_k/p^m$ with Galois
group $U$, the higher cohomology vanishes and the zero'th cohomology
is isomorphic to $T_k/p^m$.  The Lyndon-Hochschild-Serre spectral sequence
associated to $1 \to U \to \mb Z_p^\times \to (\mb Z/p^k)^\times \to
1$ then degenerates to an isomorphism
\[
H^s_c(\mb Z_p^\times; W/p^m) \cong H^s((\mb Z/p^k)^\times;
\pi_{2n}{KU}/p^m \otimes T_k).
\]
However, the map $R/p^m \to T_k/p^m$ is a Galois extension with Galois
group $(\mb Z/p^k)^\times$; it is in particular faithfully flat, and the
fixed-point functor is part of and equivalence between $R/p^m$-modules and
$T_k/p^m$-modules equipped with a semilinear Galois action.  In
particular, the higher cohomology groups vanish, and by faithfully
flat descent the $R/p^m$-module $(\pi_{2n}{KU}/p^m \otimes T_k)^{(\mb
  Z/p^k)^\times}$ is projective of rank one.  The base ring $R/p^m$ is
isomorphic to $W(\mb F_q)/p^m$, which is local, so the module is
actually free of rank one.  Moreover, the isomorphism $\pi_{2n}
\otimes_R \pi_{2m} \to \pi_{2(n+m)}$ induced by multiplication
descends to an isomorphism on invariants.

Taking limits in $m$, we find that the higher continuous cohomology
groups with coefficients in $W$ vanish, and that the module of
invariants of $W$ is free of rank one.

As a result, we find $\pi_* {\cal O}^{pre}_{\fmultmod}(R)=R[u^{\pm 1}]$
  for some unit $u$ in degree
$2$.  In particular, then, ${\cal O}^{pre}_{\fmultmod}(R)$ is
complex orientable.

The formal group $\mb G$ and the formal group associated with a
complex orientation of ${\cal O}^{pre}_{\fmultmod}(R)$ both arise from
the same descent data on $\comp{T}_p$, and so the associated formal
group is isomorphic to $\mb G$.

We now take these spectra defined on affine \'etale charts and extend
them to sheaves.  Specifically, functoriality in $R$ allows us to
sheafify, associated fibrant objects in the Jardine model structure
are functors ${\cal O}^{top}_{\multmod}$ and ${\cal
  O}^{top}_{\fmultmod}$ of ${\cal E}_\infty$-ring spectra on general
schemes over $\multmod$ and $\fmultmod$ respectively.  On any affine
$R$, the map ${\cal O}^{pre} \to {\cal O}^{top}$ is a weak equivalence
on stalks, and the presheaves of homotopy groups of ${\cal O}^{pre}$
are the quasicoherent sheaves $\omega^t$ of invariant differentials on
any affine $R$.  The sheaves associated to the homotopy groups of
${\cal O}^{top}$ are therefore the same, and the map ${\cal
  O}^{pre}(R) \to {\cal O}^{top}(R)$ is an equivalence for affine
$R$, and for a general $R$ it is recovered as a homotopy limit.
For a more detailed account of this argument in a very similar
situation see \cite[Sections 2.3-2.5]{behrens-build}.
\end{proof}

\begin{rmk}
\begin{enumerate}[i)]
\item  The above construction is motivated by the
  fact that for the $K(1)$-local (${\cal E}_\infty$-ring) spectrum
  $X={\cal O}^{top}_{\fmultmod}(R)$, one has a canonical equivalence
\[
X\stackrel{\simeq}{\to} \left(L_{K(1)}(\comp{KU}_p\wedge X)\right)^{h\mb
  Z_p^\times},
\]
\cite[Theorem 1.3]{davis-torii}, where on the right-hand side the
group $\mb Z_p^\times$ acts through the factor $\comp{KU}_p$
alone. Briefly, $X$ can be recovered from its $K$-theory. Our
construction is such that there is an equivalence of continous $\mb
Z_p^\times$-spectra
\[
L_{K(1)}(\comp{KU}_p\wedge X)\simeq L_{K(1)}(\comp{KU}_p\wedge\comp{\mb
S}_p(T)),
\]
where $\mb S_p(T)$ is built from the principal torsor on $\pi_0 X =
R$ and $\mb Z_p^\times$ on the right-hand side acts diagonally on
both factors.

\item During the above proof we saw that the \'etale site of
  $\fmultmod$ is not very big.  Specifically, all its affine objects
  are disjoint unions of maps from some $\Spf (W(\mb F_q))$ classifying
  a formal group of height one. These formal groups over $W(\mb F_q)$ in
  turn are classified by $p$-adic units $\alpha\in\mb Z_p^\times$. All
  morphisms of the site are generated by change-of-base and by
  automorphisms inducing $p$-adic Adams operations.  Existence and
  uniqueness of these realizations on affine coordinate charts is
  the content of the Goerss-Hopkins-Miller theorem, which 
  lifts the local version of Morava's construction to ${\cal
    E}_\infty$-rings. 

\item Morava considers many more forms of $K$-theory than are given by
  the values of these sheaves, Tate $K$-theory being a prominent
  example \cite[Theorem 2]{morava-forms}.  We take this as an
  opportunity to document the well-known fact that ${\cal
    E}_\infty$-realizations put very restrictive conditions on
  ramification.  The authors learned the following argument from Mike
  Hopkins.  Though these types of arguments have been known (e.g. by
  Ando, Strickland, and others) for several decades, the authors were
  unable to track down a published reference.

  Specifically, assume $p\neq 2$ and consider the extension
  $\comp{KU}_p(-)\otimes_{\mb Z_p} \mb Z_p[\zeta_p]$ which adjoins
  a $p$'th root of unity (for $p=2$ a similar argument considering
  the extension $\mb Z_2\subseteq\mb Z_2[i]$ works). Since the
  extension $\mb Z_p\to\mb Z_p[\zeta_p]$ is flat, this is a
  multiplicative cohomology theory which is a form of $K$-theory in
  Morava's sense. The corresponding homotopy-commutative ring spectrum
  does not admit an ${\cal E}_\infty$-refinement. The
  $K(1)$-local power operations would provide a lift of Frobenius
  $\varphi\co\mb Z_p[\zeta_p]\to\mb Z_p[\zeta_p]$ reducing to the
  $p$'th power map mod $(p)$.  However, the only endomorphisms of this
  ring are automorphisms, and the $p$'th power map is not injective
  mod $p$.  That a fourth root of unity cannot be adjoined to
  the sphere at $p=2$ is shown in
  \cite{schwanzl-waldhausen-vogt-roots}.
\end{enumerate}
\end{rmk}

To conclude this appendix, we explain an application of
Theorem~\ref{realizeKU} relevant to the main concern of the present
paper. Consider the generalized elliptic curve
\[
y^2 + 3xy + y = x^3
\]
over $\mb Z[1/3]$. It lies over the unramified cusp of $\overline{{\cal
  M}}_1(3)$. To identify the resulting formal group we observe that the
curve has a nodal singularity at $(-1,1)$, and that the coordinate $t =
(x+1)/(y-1)$ gives an isomorphism between the smooth locus of this
curve and $\mb P^1 \setminus \{t\ |\ t^2 + 3t + 3 = 0\}$, with
multiplication
\[
G(t, t') = \frac{tt' - 3}{t + t' + 3}
\]
and unit at $t=\infty$.  The coordinate $t^{-1}$ gives an
associated formal group law 
\[
F(x,y) = \frac{x+y+3xy}{1 - 3xy}.
\]
This is not isomorphic to the multiplicative formal group over $\mb
Z[1/3]$, but becomes so after adjoining a third root of unity
$\omega$.  It is classified by an \'etale map $f\co \mb
Z[1/3]\to\multmod$ and we can consider the form of $K$-theory
$KU^\tau$ given by ${\cal O}^{top}_{\multmod}(f)$.  Denoting
$\beta\in\pi_2KU$ the Bott element one can check
\[
KU^\tau_*=\mb Z[1/3][(\sqrt{-3}\beta)^{\pm 1}]\subseteq
KU_*\otimes_{\mb Z}\mb Z[1/3][\omega].
\]
For any $p \neq 3$, the $p$-adic completion of $KU^\tau$ is a
Lubin-Tate spectrum for its formal group law.  In the obstruction
theory of Section~\ref{sec:constructmaps}, one can equally well
substitute the ramified cusp $KU^\tau$ for $KU$ into the construction.
In place of the maps of equation~(\ref{eq:cusp-orientation}), we would
then have
\[
A \mapsto \left(\sqrt{-3} \beta\right),\,B \mapsto -\frac{1}{27}\left(\sqrt{-3} \beta\right)^3.
\]

\bibliography{level3_properties}

\end{document}